\documentclass[reqno]{amsart}

\usepackage[utf8]{inputenc}
\usepackage{amssymb}
\usepackage{amsthm}
\usepackage{amsmath}
\usepackage{tikz}

\usepackage[toc,page]{appendix}

\usepackage{enumitem}

\usepackage{tikz-cd}

\usepackage{tikz-cd}

\usepackage{tikz}
\usetikzlibrary{datavisualization}
\usetikzlibrary{datavisualization.formats.functions}

\usetikzlibrary{arrows}

\usepackage{import}

\usepackage{mathtools}
\usepackage{mathrsfs}
\usepackage{commath}

\usepackage{framed}

\usetikzlibrary{calc}

\makeatletter
\pgfarrowsdeclare{X}{X}
{
  \pgfutil@tempdima=0.3pt%
  \advance\pgfutil@tempdima by.25\pgflinewidth%
  \pgfutil@tempdimb=5.5\pgfutil@tempdima\advance\pgfutil@tempdimb by.5\pgflinewidth%
  \pgfarrowsleftextend{+-\pgfutil@tempdimb}
  \pgfarrowsrightextend{0pt}
}
{
  \pgfutil@tempdima=0.3pt%
  \advance\pgfutil@tempdima by.25\pgflinewidth%
  \pgfsetdash{}{+0pt}
  \pgfsetroundcap
  \pgfsetmiterjoin
  \pgfpathmoveto{\pgfqpoint{-5.5\pgfutil@tempdima}{-6\pgfutil@tempdima}}
  \pgfpathlineto{\pgfqpoint{5.5\pgfutil@tempdima}{6\pgfutil@tempdima}}
  \pgfpathmoveto{\pgfqpoint{-5.5\pgfutil@tempdima}{6\pgfutil@tempdima}}
  \pgfpathlineto{\pgfqpoint{5.5\pgfutil@tempdima}{-6\pgfutil@tempdima}}
  \pgfusepathqstroke 
}
\makeatother

\newtheorem{theorem}{Theorem}[section]
\newtheorem{lemma}[theorem]{Lemma}
\newtheorem{proposition}[theorem]{Proposition}
\newtheorem{corollary}[theorem]{Corollary}
\newtheorem{definition}[theorem]{Definition}
\newtheorem{example}[theorem]{Example}

\theoremstyle{remark}
\newtheorem{remark}[theorem]{Remark}
\newtheorem{question}[theorem]{Question}

\usepackage{relsize}
\usepackage{exscale}

\usepackage{mathtools}
\usepackage{mathrsfs}

\usepackage{framed}

\usepackage{datetime}

\usepackage{hyperref} 

\renewcommand{\epsilon}{\varepsilon}
\DeclareMathOperator{\EqP}{EqP}
\DeclareMathOperator{\EvP}{EvP}
\DeclareMathOperator{\Eq}{Eq}
\DeclareMathOperator{\Ev}{Ev}
\DeclareMathOperator{\Orb}{Orb}

\title[Equicontinuity, Transitivity and Sensitivity]{Equicontinuity, Transitivity and Sensitivity: the Auslander-Yorke Dichotomy Revisited}
\author{Chris Good, Robert Leek and Joel Mitchell}
\date{March 2018}

\begin{document}

\hypersetup{pageanchor=false} 

\begin{abstract}
We discuss topological equicontinuity and even continuity in dynamical systems. In doing so we provide a classification of topologically transitive dynamical systems in terms of equicontinuity pairs, give a generalisation of the Auslander-Yorke Dichotomy for minimal systems and show there exists a transitive system with an even continuity pair but no equicontinuity point. We define what it means for a system to be eventually sensitive; we give a dichotomy for transitive dynamical systems in relation to eventual sensitivity. Along the way we define a property called splitting and discuss its relation to some existing notions of chaos.
\end{abstract}

\maketitle

\hypersetup{pageanchor=true} 

Let $(X,f)$ be a discrete dynamical system, so that $f\colon X \to X$ is a (continuous) map on the metric space $X$. The dynamical system is \textit{equicontinuous at a point} $x \in X$ if for any $\epsilon>0$ there is a $\delta>0$ such that the $\delta$-ball around $x$ does not expand to more than diameter $\epsilon$ under iteration of $f$. The system itself is said to be \textit{equicontinuous} if it is equicontinuous at every point. 
Compactness of the space $X$ ensures that equicontinuity is equivalent to uniform equicontinuity: for any $\epsilon>0$ there is a $\delta>0$ such that no $\delta$-ball expands to more than diameter $\epsilon$ under iteration of $f$. Equicontinuity is extremely important in mathematical analysis where it provides the primary condition in the Arzel\`a–Ascoli theorem (see \cite[Theorem 8.2.10]{Engelking}). 
A related concept to equicontinuity is that of sensitivity. 
The system $(X,f)$ is \textit{sensitive} if every nonempty open set expands to at least diameter $\delta$ under iteration of $f$. It is obvious that the properties of sensitivity and equicontinuity are mutually exclusive. Examining the quantifiers one sees that sensitivity is \textit{almost} a negation of equicontinuity. Indeed, negating the property of equicontinuity at a given point gives a localised version of sensitivity. Auslander and Yorke \cite{AuslanderYorke} specify a type of system for which sensitivity is precisely the negation of equicontinuity: a dynamical system $(X,f)$ is said to be \textit{minimal} if the forward orbit of every point is dense in the space. 
The Auslander-Yorke dichotomy states that a compact metric minimal system is either equicontinuous or sensitive. Various analogues of this theorem have since been offered \cite{HuangKolyadaZhang}.


Topological transitivity, or simply transitivity, is a weakening of minimality. The system $(X,f)$ is said \textit{transitive} if for any nonempty open sets $U$ and $V$ there is an $n \in \mathbb{N}$ such that $f^n(U) \cap V \neq \emptyset$. Under certain conditions (compact metric being sufficient) this is equivalent to the existence of a \textit{transitive point} (i.e.~ a point with a dense orbit) \cite{AkinCarlson}. 
Transitivity and sensitivity are often cited as two key ingredients for a system to be chaotic (see, for example \cite{AuslanderYorke, Devaney}). The former prevents the system from being decomposed into multiple invariant open sets (and thereby studied as a collection of subsystems). The latter brings an element of unpredictability to the system; a small error in initial conditions may be exacerbated over time. This is clearly of particular importance in an applied setting where there is almost always going to be an error in one's measurements. In his definition of chaos, along with these two properties, Robert Devaney \cite[p.~ 50]{Devaney} included the condition that the set of periodic points be dense in whole space, thus providing ``an element of regularity'' in the midst of seemingly random behaviour. Perhaps surprisingly this regularity condition together with transitivity proved sufficient in a compact space to entail sensitivity \cite{BanksBrooksCairnsDavisStacey, GlasnerWeiss}. Since then, the gap between transitivity and sensitivity has been researched extensively (see, for example, \cite{AkinAuslanderBerg,GoodMacias,LiTuYe,Moothathu}); Akin \textit{et al} \cite{AkinAuslanderBerg} gave the following dichotomy: a compact metric transitive system is either sensitive or contains a point of equicontinuity; in $2007$, Moothathu \cite{Moothathu} generalised results in \cite{BanksBrooksCairnsDavisStacey} and \cite{AkinAuslanderBerg} by defining stronger notions of sensitivity. These variations on sensitivity have since attracted an array of interest \cite{DasSalman,HuangKolyadaZhang,Li2,WangYinYan}.

For a survey on recent developments in chaos theory, including results on sensitivity, equicontinuity and transitivity, see \cite{LiYe}.

Recently there has been a move towards studying dynamical systems without assuming the underlying space is necessarily metric or compact. To do this novel definitions were needed to generalise concepts, such as sensitivity, which \textit{prima facie} appear to be inherently metric (or at least uniform). In \cite{GoodMacias} the authors introduce what they term \textit{Hausdorff sensitivity}; they show that this coincides with the usual notion of sensitivity if the ground space is compact metric. 
Topological equicontinuity was introduced by Royden in \cite{Royden}, which, in general, is weaker than equicontinuity. The concept of even continuity, introduced by Kelley \cite[p.~ 234]{Kelley}, dates back further than topological equicontinuity and is even weaker still, although all three concepts (i.e.~ equicontinuity, topological equicontinuity and even continuity) coincide in the presence of compactness (see \cite[Theorem 7.23]{Kelley}). In contrast to equicontinuity, which is an inherently uniform concept, neither topological equicontinuity nor even continuity require the phase space to be anything more than a topological space. 
Whilst the concepts of topological equicontinuity and even continuity have gained some attention with regard to topological semigroups and families of mappings in a general setting (e.g.~ \cite{CorbachoTraieladzeVidal1,CorbachoTraieladzeVidal2}), little appears to have been done with regard to dynamical systems. 

In this paper we take a careful look at the Auslander-Yorke dichotomy via a topological approach which leads to some interesting results: After the preliminaries in Section \ref{PreDefs}, we build up some theory related to topological equicontinuity in dynamical systems in Section \ref{SectionTopEq}. Two fruits of this theory are Corollary \ref{GenA-YDichI} - a generalisation of the Auslander-Yorke dichotomy - along with an exposition, with regard to topological equicontinuity, of when a system is transitive (Theorem \ref{ChT}). Section \ref{SectionEvenCty} starts by building up theory regarding even continuity in dynamical systems. This section culminates in a construction of a compact topologically transitive system with an even continuity pair but no point of even continuity; this provides an element of regularity in a system which is Auslander-Yorke chaotic, densely and strongly Li-Yorke chaotic, but not Devaney chaotic. In Section \ref{SectionSplitting} we discuss a property we call \textit{splitting} and its relationship to topological equicontinuity, even continuity and existing notions of chaos. Finally, in Section \ref{SectEventualSensitivity} we give a dichotomy for compact transitive systems (Theorem \ref{ESD1}); they are either equicontinuous or \textit{eventually sensitive}.

Throughout this paper $X$ is a topological space. Usually it is assumed to be Hausdorff, while some results rely on the additional assumption of compactness. We will always state the relevant assumptions. 

We denote by $\mathbb{Z}$ the set of all integers; the set of positive integers $1,2,3,4,\ldots$ is denoted by $\mathbb{N}$ whilst $\mathbb{N}_0 \coloneqq\mathbb{N}\cup\{0\}$.

\section{Preliminaries}\label{PreDefs} 

\subsection{Uniform spaces}

We start by providing some background on uniformities for those who are unfamiliar. The definitions in this section can be found in \cite{Willard}. Let $X$ be a set. The diagonal of the Cartesian product $X \times X$ is the set $\Delta=\{(x,x) \mid x \in X \}$. Given two subsets $A$ and $B$ of $X \times X$, we define the composition of these sets as $A \circ B = \{(x,z) \mid \text{there exists } y \in X \text{ such that } (x,y) \in B \text{ and } (y, z) \in A\}$. We write $nA$ to denote $\underbrace{A \circ A\circ \ldots \circ A}_{n \text{ times}}$. We define the inverse $A^{-1}=\{(x,y) \mid (y,x) \in A \}$. If $A \subseteq X \times X$ contains the diagonal $\Delta$ we say it is an {\em entourage of the diagonal}.

\begin{definition} A uniformity $\mathscr{D}$ on a set $X$ is a collection of entourages of the diagonal such that the following conditions are satisfied.

\begin{enumerate}[label=\alph*.]
\item $D_1, D_2 \in \mathscr{D} \implies D_1 \cap D_2 \in \mathscr{D}$.
\item $D \in \mathscr{D}, D \subseteq E \implies E \in \mathscr{D}$.
\item $D \in \mathscr{D} \implies E \circ E \subseteq D$ for some $E \in \mathscr{D}$.
\item $D \in \mathscr{D} \implies E^{-1}\subseteq D$ for some $E \in \mathscr{D}$.
\end{enumerate}

\end{definition}

We call the pair $(X, \mathscr{D})$ a {\em uniform space}. We say $\mathscr{D}$ is {\em separating} if $\bigcap_{D \in \mathscr{D}} D = \Delta$; in this case we say $X$ is {\em separated}. A subcollection $\mathscr{E}$ of $\mathscr{D}$ is said to be a {\em base} for $\mathscr{D}$ if for any $D \in\mathscr{D}$ there exists $E \in \mathscr{E}$ such that $E \subseteq D$. Clearly any base $\mathscr{E}$ for a uniformity will have the following properties:
\begin{enumerate}
\item $D_1, D_2 \in \mathscr{D} \implies$ there exists $ E \in \mathscr{E}$ such that $ E \subseteq D_1 \cap D_2 $.
\item $D \in \mathscr{D} \implies E \circ E \subseteq D$ for some $E \in \mathscr{E}$.
\item $D \in \mathscr{D} \implies E^{-1}\subseteq D$ for some $E \in \mathscr{E}$.
\end{enumerate}
If $\mathscr{D}$ is separating then $\mathscr{E}$ will satisfy $\bigcap_{E \in \mathscr{E}} E = \Delta$. A {\em subbase} for $\mathscr{D}$ is a subcollection such that the collection of all finite intersections from said subcollection form a base. We say an entourage of the diagonal $D$ is {\em symmetric} if $D=D^{-1}$.



For an entourage $D \in \mathscr{D}$ and a point $x \in X$ we define the set $D[x]= \{y \in X \mid (x,y) \in D\}$. This naturally extends to a subset $A \subseteq X$; $D[A]= \bigcup_{x \in A}D[x]$. We emphasise that (see \cite[Section 35.6]{Willard}):
\begin{itemize}
\item For all $x \in X$, the collection $\mathscr{U}_x \coloneqq \{ D[x] \mid D \in \mathscr{D} \}$ is a neighbourhood base at $x$, making $X$ a topological space. The same topology is produced if any base $\mathscr{E}$ of $\mathscr{D}$ is used in place of $\mathscr{D}$.
\item The topology is Hausdorff if and only if $\mathscr{D}$ is separating.
\end{itemize}



A topological space is said to be \textit{Tychonoff}, or $T_{3\frac{1}{2}}$, if it is both Hausdorff and \textit{completely regular} (i.e. points and closed sets can be separated by a bounded continuous real-valued function). A topological space is Tychonoff precisely when it admits a separating uniformity. Finally we remark that for a compact Hausdorff space $X$ there is a unique uniformity $\mathscr{D}$ which induces the topology (see \cite[Section 8.3.13]{Engelking}). 

\subsection{Dynamical systems}

For those wanting a thorough introduction to topological dynamics, \cite{deVries} is an excellent resource. Most of the definitions in this section are standard and can be found there.

A {\em dynamical system} is a pair $(X,f)$ consisting of a topological space $X$ and a continuous function $f\colon X \to X$. For any $x \in X$ we denote the set of neighbourhoods of $x$ by $\mathcal{N}_x$; the elements of this set are not assumed to be open. We say the orbit of $x$ under $f$ is the set of points $\{x, f(x), f^2(x), \ldots\}$; we denote this set by $\Orb_f(x)$. We say $x$ is {\em periodic} if there exists $n \in \mathbb{N}$ such that $f^n(x)=x$; the least such $n$ is called the {\em period} of $x$; if $n=1$ we say $x$ is a \textit{fixed point}. A point $x \in X$ is {\em eventually periodic} if there exists $y \in \Orb_f(x)$ such that $y$ is periodic. It immediately follows that $\Orb_f(x)$ is finite if and only if $x$ is eventually periodic. For $x \in X$, we define the $\omega$\textit{-limit set} of $x$ under $f$, denoted $\omega_f(x)$, or simply $\omega(x)$ where there is no ambiguity, to be the set of limit points of the sequence $\left(f^n(x)\right)_{n \in \mathbb{N}}$. Formally
\[\omega_f(x)= \bigcap_{N \in \mathbb{N}}\overline{\{f^n(x) \mid n >N\}}.\]
This means that $y \in \omega_f(x)$ if and only if for every neighbourhood $U$ of $y$ and every $N \in \mathbb{N}$ there exists $n >N$ such that $f^n(x) \in U$. If $X$ is compact $\omega_f(x) \neq \emptyset$ for any $x \in X$ by Cantor's intersection theorem. Notice that $\overline{\Orb_f(x)}=\Orb_f(x) \cup \omega_f(x)$. A point $x$ is said to be \textit{recurrent} if $x \in \omega(x)$. It is said to be \textit{non-wandering} if, for any neighbourhood $U \in \mathcal{N}_x$ and any $N \in \mathbb{N}$ there is $n>N$ such that $f^n(U) \cap U \neq \emptyset$. Clearly a recurrent point is non-wandering. We define the \textit{non-wandering set} of $x$, denoted $\Omega_f(x)$, by saying that $y \in \Omega_f(x)$ if and only if for any $V \in \mathcal{N}_y$, any $U \in \mathcal{N}_x$ and any $N \in \mathbb{N}$ there exists $n >N$ such that $f^n(U) \cap V \neq \emptyset$. It follows that, for any $x \in X$, $\omega(x) \subseteq \Omega(x)$.

When $X$ is a compact Hausdorff space we will denote the unique uniformity associated with $X$ by $\mathscr{D}_X$ or usually simply $\mathscr{D}$ if there is no ambiguity. Given $A,B \subseteq X$, we denote by $N(A,B)$ the (forward) hitting times of $A$ on $B$ under $f$; specifically
\begin{equation}
    N(A,B)=\{n \in \mathbb{N} \mid f^n(A) \cap B \neq \emptyset \}.
\end{equation}
If $x \in X$ and $B \subseteq X$, we will abuse notation by writing $N(x,B)$ instead of $N(\{x\},B)$.
A dynamical system $(X,f)$ is \textit{topologically transitive}, or simply \textit{transitive}, when, for any pair of nonempty open sets $U$ and $V$, $N(U,V) \neq \emptyset$. It is {\em weakly mixing} if the product system $(X\times X, f \times f)$ is transitive. A point $x \in X$ is said to be a \textit{transitive point} if $\omega(x)=X$. A system $(X,f)$ is said to be minimal if $\omega(x)=X$ for all $x \in X$; equivalently, if there are no proper, nonempty, closed, positively-invariant subsets of $X$. (A subset $A \subseteq X$ is said to be \textit{positively invariant} (under $f$) if $f(A) \subseteq A$.)

In \cite{AkinCarlson} the authors introduce the concept of a \textit{density basis}; a density basis for a topological space $X$ is a collection $\mathcal{V}$ of nonempty open sets in $X$ such that if $A \subseteq X$ is such that $A \cap V \neq \emptyset$ for any $V \in \mathcal{V}$, then $\overline{A}=X$. 
They go on to show that if $X$ is of Baire second category (i.e. non-meagre) and has a countable density basis then topological transitivity is equivalent to the existence of a transitive point. Topologists may be more familiar with the concept of a \textit{$\pi$-base} than a density basis.

\begin{definition} A $\pi$-base for a topological space $X$ is a collection $\mathcal{U}$ of nonempty open sets in $X$ such that if $R$ is any nonempty open set in $X$ then there exists $V \in \mathcal{U}$ such that $V \subseteq R$.
\end{definition}

\begin{proposition}\label{PropPiBaseIFF}
Let $X$ be a topological space. A collection is a $\pi$-base if and only if it is a density basis.
\end{proposition}

\begin{proof}
Note first that both are defined as collections of nonempty open sets.

Suppose $\mathcal{U}$ is a $\pi$-base. Suppose $A \subseteq X$ is such that $A \cap U \neq \emptyset$ for all $U \in \mathcal{U}$. Let $W$ be open and nonempty. Then there exists $U \in \mathcal{U}$ such that $U \subseteq W$. Then $A \cap U \neq \emptyset$; therefore $A \cap W \neq \emptyset$ and so $\overline{A}=X$.

Now suppose $\mathcal{U}$ is a density basis. Assume $\mathcal{U}$ is not a $\pi$-base. Then there exists a nonempty open set $W$ such that $U\not\subseteq W$ for any $U\in \mathcal{U}$. This means that $U \setminus W \neq \emptyset$ for any $U \in \mathcal{U}$. Take \[A =\bigcup _{U \in \mathcal{U}} U \setminus W.\]
It follows that $A \cap W =\emptyset$ and, for each $U \in \mathcal{U}$, $A \cap U \neq \emptyset$. Since $\mathcal{U}$ is a density basis the latter entails $\overline{A}=X$, contradicting the fact that $A \cap W =\emptyset$. Hence $\mathcal{U}$ is a $\pi$-base.
\end{proof}


The following lemma is folklore (e.g. \cite{AkinCarlson}) and will be useful throughout.

\begin{lemma}\label{lemmaEq0.1}
Let $(X,f)$ be a dynamical system, where $X$ is a Hausdorff space. Then $(X,f)$ is topologically transitive if and only if $N(U,V)$ is infinite for any pair of nonempty open sets $U$ and $V$.
\end{lemma}

\begin{remark}\label{remarkOmegaTrans}
It follows from Lemma \ref{lemmaEq0.1} that, for a transitive system $(X,f)$ where $X$ is a Hausdorff space, we have $\Omega(x)=X$ for any $x \in X$.
\end{remark}

For the rest of this section $X$ is a Tychonoff space.

Suppose $\mathscr{D}$ is a compatible uniformity for $X$. Let $U \subseteq X$ and let $D \in \mathscr{D}$ be symmetric. Define 
\begin{equation}
    N_D(U)=\{n \in \mathbb{N} \mid \exists x,y \in U \text{ such that } (f^n(x),f^n(y))\notin D \}.
\end{equation}
We say a system exhibits \textit{sensitive dependence on initial conditions} (or is \textit{sensitive}) if there exist a compatible uniformity $\mathscr{D}$ and a symmetric $D \in \mathscr{D}$ such that $N_D(U) \neq \emptyset$ for any nonempty open $U \subseteq X$. In this case we say $D$ is a sensitivity entourage $(X,f)$. 
If $X$ is a metric space, for $U \subseteq X$ and $\delta>0$ we define 
\begin{equation}
    N_\delta(U)=\{n \in \mathbb{N} \mid \exists x,y \in U \text{ such that } d(f^n(x),f^n(y))\geq \delta \}.
\end{equation}
In this case we say the system is sensitive if there exists $\delta>0$ such that $N_\delta(U) \neq \emptyset$ for any nonempty open set $U$. The definitions for a metric space coincide when it is equipped with the metric uniformity (see \cite{GoodMacias}). We invite readers unfamiliar with uniformities to notice the similarities in these definitions; it may be helpful for such readers to view the statement, ``there exists $D\in\mathscr{D}$ such that $(x,y) \in D$,'' as, ``there exists $\delta>0$ such that $d(x,y) < \delta$''. Similarly ``$(x,y) \notin D$" may be read as ``$d(x,y) \geq \delta$". In this way, $D[x]$ may be thought of as $B_\delta(x)$. The uniform structure of a space can be used to mimic existing metric proofs (see, for example, \cite{GoodMacias}). In the proof of the following lemma, which is folklore, we invite the reader to observe how entourages have simply replaced the real numbers which would have designated distances for a metric version.


\begin{lemma}\label{R0.1} If $X$ is a Tychonoff space and $(X,f)$ is a sensitive dynamical system, with sensitivity $D \in \mathscr{D}$, where $\mathscr{D}$ is a compatible uniformity for $X$, then for any nonempty open $U \subseteq X$ the set $N_D(U)$ is infinite.
\end{lemma}

\begin{proof}
Let $U \subseteq X$ be nonempty open and suppose $N_D(U)$ is finite; let $k\in \mathbb{N}$ be an upper bound for this set. Let $E \in \mathscr{D}$ be such that $2E \subseteq D$. Let $x \in U$. By continuity we may choose a symmetric entourage $D_0 \in \mathscr{D}$ such that, for any $y \in X$, if $(x,y) \in D_0$ then $(f^i(x),f^i(y)) \in E$ for all $i \in \{1,\ldots , k\}$. Consider the set $W \coloneqq U \cap D_0[x]$; $W$ is a neighbourhood of $x$. Thus $N_D(W) \neq \emptyset$ by sensitivity, but $f^i(W) \subseteq E[f^i(x)]$ for $i \in \{1, \ldots, k\}$; in particular if $y,z \in W$ then $(f^i(y),f^i(z))\in D$ for $i \in \{0, 1, \ldots, k\}$. Therefore there exists $n>k$ and $y,z \in W$ such that $(f^n(y),f^n(z))\notin D$. As $W \subseteq U$ we have a contradiction and the result follows.
\end{proof}

Let $X$ now be a compact Hausdorff space. A point $x \in X$ is said to be an \textit{equicontinuity point} of the system $(X,f)$ if 
\begin{equation} \forall E \in \mathscr{D} \, \exists D \in \mathscr{D}: \forall n \in \mathbb{N}, \,  y \in D[x] \implies f^n(y) \in E[f^n(x)].
\end{equation}
In this case we say $(X,f)$ is \textit{equicontinuous at} $x$. If $(X,f)$ is equicontinuous at every $x \in X$ then we say the system itself is \textit{equicontinuous}. When $X$ is compact this is equivalent to the system being \textit{uniformly equicontinuous}, that is
\begin{equation}\label{eqn0.6} \forall E \in \mathscr{D} \, \exists D \in \mathscr{D}: \forall n \in \mathbb{N}, \,  (x,y) \in D \implies (f^n(x),f^n(y)) \in E.
\end{equation}
We denote the set of all equicontinuity points by $\Eq(X,f)$, so a system is equicontinuous if $\Eq(X,f)=X$. 




The following results will be useful; versions for compact metric systems may be found in \cite{AkinAuslanderBerg}, their proofs may be mimicked to give the following more general versions.

\begin{lemma}\textup{\cite{AkinAuslanderBerg}}\label{lemma1Omega(x)omega(x)} Let $(X,f)$ be a dynamical system, where $X$ is a compact Hausdorff space. If $x \in \Eq(X,f)$ then $\omega_f(x) = \Omega_f(x)$.
\end{lemma}

\begin{theorem}\label{thmEqPtsCoincideWithTransPts}\textup{\cite{AkinAuslanderBerg}}
Let $(X,f)$ be a transitive dynamical system where $X$ is a compact Hausdorff space. If $\Eq(X,f) \neq \emptyset$ then the set of equicontinuity points coincide with the set of transitive points.
\end{theorem}

\begin{corollary}
Let $(X,f)$ be a transitive dynamical system where $X$ is a compact Hausdorff space. If $X$ is not separable then $\Eq(X,f)=\emptyset$.
\end{corollary}

\begin{proof}
If $\Eq(X,f) \neq \emptyset$ then every equicontinuity point is a transitive point. If the system has a transitive point then it has a countable dense subset and is thereby separable.
\end{proof}

We end this section of the preliminaries with three common notions of chaos. A dynamical system is said to be Auslander-Yorke chaotic (see \cite{AuslanderYorke}) if it is both transitive and sensitive. If, in addition, it has a dense set of periodic points it is said to be Devaney chaotic (see \cite{Devaney}).

If $X$ is a metric space and $(X,f)$ a dynamical system, then we say a pair $(x,y) \in X \times X$ is \textit{proximal} if \[\liminf_{n\rightarrow \infty} d(f^n(x),f^n(y))=0,\] 
and \textit{asymptotic} if \[\limsup_{n\rightarrow \infty} d(f^n(x),f^n(y))=0.\] 
The pair $(x,y)$ is said to be a \textit{Li-Yorke pair} if they are proximal but not asymptotic. It is said to be a \textit{strong Li-Yorke pair} if it is both a Li-Yorke pair and recurrent in the product system $(X^2, f\times f)$. A set $S \subseteq X$ is said to be \textit{scrambled} if every pair of distinct points in $S$ form a Li-Yorke pair; it is said to be \textit{strongly scrambled} if every pair of distinct points in $S$ form a strong Li-Yorke pair. A system $(X,f)$ is said to be \textit{Li-Yorke chaotic} (see \cite{LiYorke}) if there exists an uncountable scrambled set $S$. If $S$ is strongly scrambled we say $(X,f)$ is \textit{strongly Li-Yorke chaotic}. Finally if $S$ is dense in $X$ then we say the system is \textit{densely Li-Yorke chaotic} \cite[Section 7.3]{deVries}.

\subsection{Shift spaces}\label{SectionShiftSpaces}

Given a finite set $\Sigma$ considered with the discrete topology, \textit{the full one sided shift with alphabet }$\Sigma$ consists of the set of infinite sequences in $\Sigma$, that is $\Sigma^{\mathbb{N}_0}$, which we consider with the product topology. This forms a dynamical system with the \textit{shift map} $\sigma$, given by \[\sigma\left(\langle x_i \rangle _{i\geq 0}\right)=\langle x_i \rangle _{i \geq 1}. \] 
A \textit{shift space} is some compact positively-invariant (under $\sigma$) subset of some full shift. Let $X$ be a shift space, with alphabet $\Sigma$. Given a finite word, $a_0 a_1\ldots a_m$, made up of elements of $\Sigma$, we denote by $[a_0a_1 \ldots a_m]$ the \textit{cylinder set }induced by the word $a_0 a_1\ldots a_m$; this is all points in $X$ which begin with `$a_0 a_1\ldots a_m$'. The collection of all cylinder sets intersected with $X$ form a base for the induced subspace topology from the Tychonoff product $\Sigma ^{\mathbb{N}_0}$. For a symbol $a \in \Sigma$, we use the notation $a^n$, for some $n \in \mathbb{N}$, to mean
\[\underbrace{aaa\ldots a}_{\text{$n$ times.}}\]
For a word $W$, we use $\abs{W}$ to denote the length of $W$. So if $W=w_0w_1w_2\ldots w_n$, then $\abs{W}=n+1$. For the word $W$, we refer to the set $\{w_kw_{k+1}\ldots w_{k+j}\mid 0 \leq k \leq n, \, 0 \leq j \leq n-k \}$ as the \textit{set of all subwords of} $W$; the elements of this set are called subwords of $W$. We refer to any subword of the form $w_0w_1\ldots w_k$, for some $k\leq n$, as an \textit{initial segment} of $W$. In similar fashion, if $x= \langle x_i \rangle _{i \geq 0}\in \Sigma^{\mathbb{N}_0}$ and $n \in \mathbb{N}_0$, we refer to $x_0x_1\ldots x_n$ as an \textit{initial segment} of $x$.

For those wanting more information about shift systems, \cite[Chapter 5]{deVries} provides a thorough introduction to the topic.

\section{Topological equicontinuity and the Auslander-Yorke Dichotomy}\label{SectionTopEq}


Previously we defined equicontinuity for compact Hausdorff dynamical systems. More generally \cite{Willard}, if $X$ is any topological space and $Y$ a uniform space, we say that a family $\mathscr{F}$ of continuous functions from $X$ to $Y$ is {\em equicontinuous at $x \in X$} if for each $E \in \mathscr{D}_Y$ there exists $U \in \mathcal{N}_x$ such that, for each $f \in \mathscr{F}$, $f(U) \subseteq E[f(x)]$. We say $\mathscr{F}$ is {\em equicontinuous} provided it is equicontinuous at each point of $X$. To generalise this to arbitrary spaces, Royden \cite{Royden} presents the following concept of \textit{topological equicontinuity}. If $X$ and $Y$ are topological spaces we say a collection of maps $\mathscr{F}$ from $X$ to $Y$ is \textit{topologically equicontinuous} at an ordered pair $(x,y) \in X \times Y$ if for any $O \in \mathcal{N}_y$ there exist neighbourhoods $U \in \mathcal{N}_x$ and $V\in \mathcal{N}_y$ such that, for any $f \in \mathscr{F}$, if $f(U) \cap V \neq \emptyset$ then $f(U) \subseteq O$; when this is the case we refer to $(x,y)$ as an \textit{equicontinuity pair}. We say $\mathscr{F}$ is \textit{topologically equicontinuous} at a point $x \in X$ if it is topologically equicontinuous at $(x,y)$ for all $y\in Y$. We say the collection is \textit{topologically equicontinuous} if it is topologically equicontinuous at every $x \in X$.
If $(x,y)$ is an equicontinuity pair then we will say $y$ is an {\em equicontinuity partner} of $x$.

Topological equicontinuity and the usual notion of equicontinuity coincide when $Y$ is a compact Hausdorff space. 

\begin{theorem}\textup{\cite[p. 364]{Royden}}\label{EquivalencesOfTopEqAndEq} Let $X$ and $Y$ be topological spaces, with $\mathscr{F}$ a collection of continuous functions from $X$ to $Y$. Let $x \in X$. If $Y$ is a Tychonoff space and $\mathscr{F}$ is equicontinuous at $x$ then $\mathscr{F}$ is topologically equicontinuous at $x$. If $Y$ is a compact Hausdorff space then the collection $\mathscr{F}$ is equicontinuous at $x \in X$ if and only if it is topologically equicontinuous at $x$.
\end{theorem}

If $(X,f)$ is a dynamical system, we will denote the set of equicontinuity pairs by $\EqP(X,f)$. Note that in this case, if we consider the above definitions, we have $Y=X$ and $\mathscr{F}=\{f^n \mid n \in \mathbb{N}\}$. By definition it follows that $(X,f)$ is topologically equicontinuous precisely when $\EqP(X,f) = X \times X$. For $(x,y) \in \EqP(X,f)$, we refer to the condition
\begin{equation}\label{eqnEqCondition}
\forall O\in \mathcal{N}_y \, \exists U \in \mathcal{N}_x \, \exists V \in \mathcal{N}_y : \forall n \in \mathbb{N}, f^n(U) \cap V \neq \emptyset \implies f^n(U) \subseteq O,
\end{equation}
as the \textit{topological equicontinuity condition} for $x$ and $y$. We say that $U$ and $V$, as in Equation \ref{eqnEqCondition}, satisfy the topological equicontinuity condition for $x$, $y$ and $O$.



The following simple observation relies solely on continuity and will be useful throughout what follows.

\begin{lemma}\label{lemmaS}
Let $(X,f)$ be a dynamical system, where $X$ is a Hausdorff space. Let $x,y \in X$ and $n \in \mathbb{N}$. Pick $O \in \mathcal{N}_y$ and let
\[S=\{k \in \{1,\ldots, n \} \mid f^k(x)=y\}.\]
There exist neighbourhoods $U \in \mathcal{N}_x$ and $V \in \mathcal{N}_y$ such that $N(U,V) \cap\{1, \ldots, n\}=S$
and $f^k(U) \subseteq V \subseteq O$ for all $k \in S$.
\end{lemma}

\begin{proof}
Let $S=\{ k \in \{1,\ldots,n\} \mid f^k(x)=y\}$ (this set may be empty). For all $i \in \{1, \ldots, n\}\setminus S$, let $U_i \in \mathcal{N}_{f^i(x)}$ and $V_i \in \mathcal{N}_y$ be such that $U_i \cap V_i =\emptyset$. Define
\[V\coloneqq \left(\bigcap _{i \in \{1, \ldots, n\}\setminus S} V_i \right) \cap O.\]
Then $V \in \mathcal{N}_y$. Now take 
\[U \coloneqq \left(\bigcap _{i\in  \{1, \ldots, n\}\setminus S} f^{-i}\left(U_i\right)\right) \cap \left(\bigcap _{i\in  S} f^{-i}\left(V\right) \right).\]
Notice $U \in \mathcal{N}_x$.

By construction, $N(U, V)\cap\{1, \ldots, n\}=S$ and $f^k(U) \subseteq V \subseteq O$ for all $k \in S$.
\end{proof}
In particular Lemma \ref{lemmaS} shows that any pair $(x,y) \in X \times X$, satisfy the following weakened version of the topological equicontinuity condition (Equation \ref{eqnEqCondition}).

\begin{corollary}\label{corFiniteTopEqCondition}
Let $(X,f)$ be a dynamical system, where $X$ is a Hausdorff space. Let $x,y \in X$ and $n \in \mathbb{N}$. Then for any $O \in \mathcal{N}_y$ there exist $U \in \mathcal{N}_x$ and $V \in \mathcal{N}_y$ such that, for any $k \in \{1,\ldots, n\}$, \[f^k(U) \cap V \neq \emptyset \implies f^k(U) \subseteq O.\]
\end{corollary}
\begin{proof}
Immediate from Lemma \ref{lemmaS}.
\end{proof}

If $(X,f)$ is a Hausdorff dynamical system and the points $x, y \in X$ are such that there exist neighbourhoods $U \in \mathcal{N}_x$ and $V \in \mathcal{N}_y$ such that, for all $n\in \mathbb{N}$, $f^n(U) \cap V = \emptyset$ then $(x,y) \in \EqP(X,f)$; this is vacuously true. The following result adds to this.

\begin{proposition}\label{prop1.0}
Let $(X,f)$ be a dynamical system where $X$ is Hausdorff  space. Let $x,y \in X$ and suppose that $y \notin \Omega(x)$. Then $(x,y) \in \EqP(X,f)$.

\end{proposition}


\begin{proof}
Let $O \in \mathcal{N}_y$. Take $U\in \mathcal{N}_x$, $V\in \mathcal{N}_y$ and $N \in \mathbb{N}$ such that $f^n(U) \cap V =\emptyset$ for all $n>N$. 
By Corollary \ref{corFiniteTopEqCondition}, there exist $U^\prime$ and $V^\prime$ such that, for any $k \in \{1,\ldots, N\}$, if $f^k(U^\prime) \cap V^\prime \neq \emptyset$ then $f^k(U^\prime) \subseteq O$; without loss of generality $U^\prime \subseteq U$ and $V^\prime \subseteq V \cap O$. Then, since $f^n(U^\prime) \cap V^\prime = \emptyset$ for all $n>N$, $U^\prime$ and $V^\prime$ satisfy the topological equicontinuity condition for $x$, $y$ and $O$. As $O \in \mathcal{N}_y$ was picked arbitrarily the result follows.
\end{proof}

With this in mind we make the following definition.

\begin{definition}
If $(X,f)$ is a dynamical system, where $X$ is a Hausdorff space. We say $(x,y) \in X \times X$ is a trivial equicontinuity pair if $y \notin \Omega(x)$.
\end{definition}

    \begin{remark} Proposition \ref{prop1.0} tells us that a trivial equicontinuity pair is indeed an equicontinuity pair.
    \end{remark}

Generally, in a non-compact Tychonoff space, topologically equicontinuity, whilst clearly necessary for equicontinuity (Theorem \ref{EquivalencesOfTopEqAndEq}), is not sufficient; it is a strictly weaker property than equicontinuity. Example \ref{ex1.1} shows this. First, recall that a metric system $(X,f)$ is said to be \textit{expansive} if there exists $\delta>0$ such that for any $x$ and $y$, with $x \neq y$, there exists $k \in {\mathbb{N}_0}$ such that $d(f^k(x),f^k(y))\geq \delta$. It is easy to see that if $X$ is perfect (i.e. without isolated points) then expansivity implies sensitivity.

\begin{example}\label{ex1.1}
Consider the dynamical system $(X=\mathbb{R}\setminus\{0\}, f)$, where $f(x)=2x$. Using Proposition \ref{prop1.0} it can be verified that $\EqP(X,f)=X\times X$, hence the system is topologically equicontinuous. However this system is not only sensitive but it is also expansive. Each of these properties (the latter, since $X$ is perfect) are mutually exclusive with the existence of an equicontinuity point, thus $x \notin \Eq(X,f)$ for any $x \in X$. 
\end{example}

\begin{lemma}\label{lemma1.0}
Let $(X,f)$ be a dynamical system, where $X$ is a Hausdorff space. If $(x,y) \in \EqP(X,f)$ then either they are a trivial equicontinuity pair or $y \in \omega(x)$.
\end{lemma}

\begin{proof} Suppose $(x,y)$ is a non-trivial equicontinuity pair (otherwise we are done). Now suppose $y\notin \omega(x)$; then there exists $O \in \mathcal{N}_y$ and $N \in \mathbb{N}$ such that for all $n>N$ we have $f^n(x) \notin O$. Since $(x,y)$ are a nontrivial pair, for any neighbourhoods $U$ and $V$ of $x$ and $y$ respectively, the set $N(U,V)$ is infinite. Pick $U \in \mathcal{N}_x$ and $V \in \mathcal{N}_y$ and let $n>N$ be such that $f^n(U) \cap V \neq \emptyset$. Then, as $f^n(x) \notin O$, $f^n(U) \not\subseteq O$. As $U$ and $V$ were arbitrary neighbourhoods this contradicts the fact that $(x,y) \in \EqP(X,f)$.
\end{proof}

This means that a pair $(x,y)$ is a non-trivial equicontinuity pair if and only if it is an equicontinuity pair and $y \in \omega(x)$.


The statement $(x,y) \notin \EqP(X,f)$, for $x,y \in X$, means precisely
\begin{equation}\label{eqn1.1}
    \exists O \in \mathcal{N}_y : \forall U \in \mathcal{N}_x \, \forall V \in \mathcal{N}_y \, \exists n \in \mathbb{N} : f^n(U) \cap V \neq \emptyset \text{ and } f^n(U) \not\subseteq O.
\end{equation}
In particular, for any pair of neighbourhoods $U\in \mathcal{N}_x$ and $V\in \mathcal{N}_y$, we have that $U$ meets $V$ after some number of iterations of $f$. If $N(U, V)$ were finite, for some such pair, then $(x,y) \in \EqP(X,f)$ by Proposition \ref{prop1.0} (it would be a trivial equicontinuity pair), thus $N(U,V)$ is infinite. By definition this means that $y \in \Omega(x)$. (NB. We shall refer to a neighbourhood such as $O$ in Equation (\ref{eqn1.1}) as a \textit{splitting neighbourhood} of $y$ with regard to $x$.) This leads us to the following generalisation of Lemma \ref{lemma1Omega(x)omega(x)}.

\begin{lemma}\label{lemma2Omega(x)omega(x)}
Let $(X,f)$ be a dynamical system where $X$ is a Hausdorff space. If $(X,f)$ is topologically equicontinuous at $x \in X$ then $\omega(x)=\Omega(x)$.
\end{lemma}
\begin{proof}
Pick $y \in X$ arbitrarily; note that $(x,y) \in \EqP(X,f)$ by hypothesis. Since $\omega(x) \subseteq \Omega(x)$ if suffices to consider the case when $y \in \Omega(x)$. In this case we have $(x,y)$ is a non-trivial equicontinuity pair. Hence $y \in \omega(x)$ by Lemma \ref{lemma1.0}.
\end{proof}

We are now in a position to characterise transitive dynamical systems on Hausdorff spaces purely with reference to equicontinuity pairs.

\begin{theorem}\label{ChT} Let $X$ be a Hausdorff space, and let $f \colon X \to X$ be a continuous function. Then $(X,f)$ is a transitive dynamical system if and only if there are no trivial equicontinuity pairs.
\end{theorem}

\begin{proof} Suppose first that $(X,f)$ is transitive. Let $(x,y) \in X \times X$ be given and let $U \in \mathcal{N}_x$ and $V \in \mathcal{N}_y$. By transitivity, $N(U,V)$ is infinite (see Lemma \ref{lemmaEq0.1}). Since $U$ and $V$ were arbitrary neighbourhoods it follows that $(x,y)$ is not a trivial equicontinuity pair.

Now suppose $(X,f)$ has no trivial equicontinuity pairs and let $U$ and $V$ be nonempty open sets. Pick $x \in U$ and $y \in V$; $(x,y)$ is a not a trivial equicontinuity pair. If $(x,y) \in \EqP(X,f)$ then, by Lemma \ref{lemma1.0}, $y \in \omega(x)$ from which is follows that $N(U,V) \neq \emptyset$. If $(x,y) \notin \EqP(X,f)$ then by Equation (\ref{eqn1.1}) there exists $n \in N(U,V)$. In every case, $N(U,V) \neq \emptyset$ and we have transitivity.
\end{proof}

The following corollary is a direct consequence of putting Lemma \ref{lemma1.0} and Theorem \ref{ChT} together.

\begin{corollary}\label{Cor1.Eq} Let $X$ be a Hausdorff space and $(X,f)$ be a transitive dynamical system. If $(x,y) \in \EqP(X,f)$ then $y \in \omega(x)$.
\end{corollary}

We now construct a class of examples which have no isolated points and non-trivial equicontinuity pairs but no points of topological equicontinuity. The information provided on shift spaces in Section \ref{SectionShiftSpaces} will be of relevance here.

\begin{example}\label{ExampleNonTrivEqPairNoEqPt}
Take $\Sigma =\{0,1,2,\ldots,m\}$, where $m\geq 2$. For each $k \in \mathbb{N}$, let $W_k$ represent a word of length $k$ containing only the symbols $\{1, 2, \ldots, m\}$. Let $\mathcal{W}$ be the collection of all sequences of the form:
\[W_10W_20^2W_30^3\ldots0^{n-1}W_n0^n\ldots,\]
Now take \[Y= \mathcal{W} \cup \{0^nx \mid x \in \mathcal{W}, n \in \mathbb{N}\} \cup \{0^\infty\},\]
and let
\[X\coloneqq \overline{\{\sigma^k(y) \mid y \in Y, k \in \mathbb{N}_0\}}.\]
where the closure is taken with regard to the full shift $\Sigma^\omega$. It is worth observing that the $\omega$-limit sets of points in $\mathcal{W}$ are points of the following forms:
\[0^\infty \text{ and } W_k0^\infty.\] 
Notice that, for any $x \in \mathcal{W}$, $n \in \mathbb{N}$ and $k \in {\mathbb{N}_0}$, $\sigma^k(x) \in [0^n]$ if and only if, for all $y \in \mathcal{W}$, $\sigma^k(y) \in [0^n]$. With this observation in mind, we claim that if $x \in \mathcal{W}$, then $(x,0^\infty)\in \EqP(X,f)$. Indeed, pick such an $x$; write $x=W_10W_20^2W_30^3\ldots$. Now let $O \ni 0^\infty$ be open. Let $V=[0^n] \subseteq O$ and take $U=[W_10W_2]$. If $y\in U$ then $y \in \mathcal{W}$ 
by construction. But by our observation, if $\sigma^k(y) \in V$ then $\sigma^k\left(\mathcal{W}\right) \subseteq V \subseteq O$. Hence $(x,0^\infty)\in \EqP(X,f)$.

It remains to observe that $\Eq(X,f)=\emptyset$, because shift systems, with no isolated points, are sensitive. By Theorem \ref{EquivalencesOfTopEqAndEq}, this means there are no points of topological equicontinuity.
\end{example} 
Example \ref{ExampleNonTrivEqPairNoEqPt} demonstrates that, even in a compact metric setting, a point may have non-trivial equicontinuity partners but not be a point of equicontinuity.


We will now build up some results relating to equicontinuity pairs in dynamical systems, this will culminate in a generalisation of the Auslander-Yorke Dichotomy.

\begin{lemma}\label{lemmaEqPOpenAt}
Let $(X,f)$ be a dynamical system, where $X$ is a Hausdorff space. Let $x,y \in X$. If $(x,y) \in \EqP(X,f)$, $f$ is open at $y$ and there is a neighbourhood base for $y$, $\mathcal{B}_y\subseteq \mathcal{N}_y$, such that $f^{-1}(f(O))=O$ for all $O \in \mathcal{B}_y$, then $(x,f(y)) \in \EqP(X,f)$.
\end{lemma}

\begin{proof}
Let $O \in \mathcal{N}_{f(y)}$. Then $f^{-1}(O) \in \mathcal{N}_y$. Let $O^\prime \in \mathcal{B}_y$ be such that $f(O^\prime)\subseteq O$. Notice that, since $f$ is open at $y$, $f(O^\prime) \in \mathcal{N}_{f(y)}$. Since $(x,y) \in \EqP(X,f)$ there exist $U \in \mathcal{N}_x$ and $V \in \mathcal{N}_y$ satisfying the topological equicontinuity for $x$, $y$ and $O^\prime$; without loss of generality $V \subseteq O^\prime$ and $V \in \mathcal{B}_y$.  If $x \neq y$ then, without loss of generality, $U \cap V =\emptyset$. If $x = y$ then, without loss of generality $U=V$. Because $f$ is open at $y$, $f(V) \in \mathcal{N}_{f(y)}$. For any $n \in \mathbb{N}$, if $f^n(U) \cap f(V) \neq \emptyset$ then $f^{n-1}(U) \cap f^{-1}(f(V)) \neq \emptyset$. Because $V \in \mathcal{B}_y$ we have $f^{-1}(f(V))=V$, hence $f^{n-1}(U) \cap V \neq \emptyset$. If $n=1$ then it follows that $U=V$ and so $U \subseteq O^\prime$. This itself implies $f(U) \subseteq f(O^\prime) \subseteq O$. If $n>1$ then $f^{n-1}(U) \subseteq  O^\prime$ by topological equicontinuity at $x$ and $y$. This implies $f^n(U) \subseteq f(O^\prime) \subseteq O$. 
\end{proof}

\begin{corollary}
Let $(X,f)$ be a dynamical system, where $X$ is a Hausdorff space. If $f$ is a homeomorphism and $(x,y) \in \EqP(X,f)$ then $(x,f(y)) \in \EqP(X,f)$.
\end{corollary}

\begin{proof}
Immediate from Lemma \ref{lemmaEqPOpenAt}.
\end{proof}

\begin{lemma}\label{RemarkInf1}  Let $(X,f)$ be a dynamical system, where $X$ is Hausdorff space. Suppose $(x,y) \notin \EqP(X,f)$ and let $O$ be a splitting neighbourhood of $y$ with regard to $x$. Then, for any pair of neighbourhoods $U$ and $V$ of $x$ and $y$ respectively, the set of natural numbers $n$ for which $f^n(U) \cap V \neq \emptyset$ and $f^n(U) \not\subseteq O$ is infinite.
\end{lemma}

\begin{proof}
The proof is similar to that of Proposition \ref{prop1.0}.

Let $U \in \mathcal{N}_x$ and $V \in \mathcal{N}_y$. Take 
\[A= \{n \in \mathbb{N} \mid f^n(U) \cap V \neq \emptyset \text{ and } f^n(U) \not\subseteq O \}.\]
Suppose that $A$ is finite; note that $A \neq \emptyset$ as $(x,y) \notin \EqP(X,f)$. Let $N$ be the largest element in $A$. By Corollary \ref{corFiniteTopEqCondition}, there exist $U^\prime \in \mathcal{N}_x$ and $V^\prime \in \mathcal{N}_y$ such that, for any $k \in \{1,\ldots, N\}$, if $f^k(U^\prime) \cap V^\prime \neq \emptyset$ then $f^k(U^\prime) \subseteq O$; without loss of generality $U^\prime \subseteq U$ and $V^\prime \subseteq V$. But as $(x,y) \notin \EqP(X,f)$ we have $A^\prime \neq \emptyset$, where
\[A^\prime= \{n \in \mathbb{N} \mid f^n(U^\prime) \cap V^\prime \neq \emptyset \text{ and } f^n(U^\prime) \not\subseteq O \}.\]
Thus there exists $m>N$ with $m \in A^\prime \subseteq A$.
\end{proof}

\begin{lemma}\label{OrbNonEqPair}
Let $(X,f)$ be a dynamical system, where $X$ is Hausdorff space. Let $x,y,z \in X$ and let $z \in \overline{\Orb(x)}$. If $(x,y) \notin \EqP(X,f)$ and $O$ is a splitting neighbourhood of $y$ with regard to $x$ then $(z,y)\notin \EqP(X,f)$ and $O$ is a splitting neighbourhood of $y$ with regard to $z$.
\end{lemma}

\begin{proof}
Let $U \in \mathcal{N}_{z}$ and $V \in \mathcal{N}_y$. Let $n \in \mathbb{N}$ be such that $W=f^{-n}(U) \ni x$. Take $m>n$ such that $f^m(W) \cap V \neq \emptyset$ and $f^m(W) \not \subseteq O$; such an $m$ exists by Lemma \ref{RemarkInf1}.
\end{proof}

\begin{remark}\label{RemarkOrbEqPair}
The contrapositive of Lemma \ref{OrbNonEqPair} is: If $(y,z) \in \EqP(X,f)$ and $y \in \overline{\Orb(x)}$ then $(x, z)\in \EqP(X,f)$.
\end{remark}

\begin{corollary}\label{CorNonTrivEqPairIsTrans}
Let $(X,f)$ be a Hausdorff dynamical system and suppose $x,y, z \in X$. If $(y,z)$ is a trivial (resp. non-trivial) equicontinuity pair and $y \in \overline{\Orb(x)}$ then $(x,z)$ is a trivial (resp. non-trivial) equicontinuity pair. In particular, if $(x,y)$ and $(y,z)$ are non-trivial equicontinuity pairs then so is $(x,z)$.
\end{corollary}
\begin{proof}
If $(y,z)$ is a trivial equicontinuity pair then $z \notin \Omega(y)$. Let $U \in \mathcal{N}_y$, $V \in \mathcal{N}_z$ and $N \in \mathbb{N}$ be such that, for any $n>N$, $f^n(U) \cap V =\emptyset$. Now let $m  \in {\mathbb{N}_0}$ be such that $W=f^{-m}(U) \ni x$. Then, for all $n>N+m$, $f^n(W) \cap V =\emptyset$. Thus $(x,z)$ is a trivial equicontinuity pair. 

Now suppose that $(y,z)$ is a non-trivial equicontinuity pair.
If $y \in \Orb(x)$ then $\omega(x)=\omega(y)$ and so $z \in \omega(x)$. If $y \in \omega(x)$ then, since $\omega$-limit sets are positively invariant, $z \in \omega(x)$. Therefore we have $z \in \omega(x)$. It now suffices to check $(x,z) \in \EqP(X,f)$; but this is just Remark \ref{RemarkOrbEqPair}.

Finally, if $(x,y)$ and $(y,z)$ are non-trivial equicontinuity pairs then $y \in \omega(x)$ and the result follows by the above.
\end{proof}

\begin{remark}
Corollary \ref{CorNonTrivEqPairIsTrans} shows that the relation given by `non-trivial equicontinuity pair' is transitive.
\end{remark}

\begin{remark}\label{RemarkEqPairImpliesTransPtHasEqPartner}
It follows from Corollary \ref{CorNonTrivEqPairIsTrans} that if a system has a transitive point, say $x$, then, if $(a,b)$ is an equicontinuity pair then $(x,b)$ is also an equicontinuity pair; every equicontinuity partner is an equicontinuity partner of the transitive point.
\end{remark}


\begin{corollary}
Let $X$ be a Hausdorff space. If $(X,f)$ is minimal then, for any $x,y \in X$,
\[(x,y) \in \EqP(X,f) \implies \forall z \in X, \, (z,y) \in \EqP(X,f),\]
and
\[(x,y) \notin \EqP(X,f) \implies \forall z \in X, \, (z,y) \notin \EqP(X,f).\]
\end{corollary}
\begin{proof}
The former statement follows from Corollary \ref{CorNonTrivEqPairIsTrans}, the latter from Lemma \ref{OrbNonEqPair}.
\end{proof}


The following theorem is a generalisation of \cite[Theorem 2.4]{AkinAuslanderBerg} (see Theorem \ref{thmEqPtsCoincideWithTransPts}).

\begin{theorem}\label{thmTopEqPtsCoincideWithTransPts}
Let $(X,f)$ be a transitive dynamical system, where $X$ is a Hausdorff space. Suppose there exists a topological equicontinuity point. Then the set of topological equicontinuity points coincides with the set of transitive points.

In particular, if $(X,f)$ is a minimal system and there is a topological equicontinuity point then the system is topologically equicontinuous.
\end{theorem}

\begin{proof}
Let $x \in X$ be a point of topological equicontinuity. By Lemma \ref{lemma2Omega(x)omega(x)}, $\omega(x)=\Omega(x)$; but since $(X,f)$ is a transitive system $\Omega(x)=X$ by Remark \ref{remarkOmegaTrans}. Hence $x$ is a transitive point.

Now suppose $x$ is a transitive point. Let $y$ be a point of topological equicontinuity. Then $y \in \omega(x)$ as $x$ is a transitive point. Now, $(y,z) \in \EqP(X,f)$ for all $z \in X$, and these are all non-trivial equicontinuity pairs by Theorem \ref{ChT}, therefore, by Corollary \ref{CorNonTrivEqPairIsTrans}, it follows that $(x,z) \in \EqP(X,f)$ for all $z \in X$; i.e. $x$ is a point of topological equicontinuity. 
\end{proof}

We are now in a position to present a generalised version of the Auslander-Yorke dichotomy for minimal systems; in \cite{AuslanderYorke} the authors show that a compact metric minimal system is either equicontinuous or is sensitive. The following definition was given by Good and Mac\'ias in \cite{GoodMacias}; they show it is equivalent to sensitivity if $X$ is a compact Hausdorff space. 

\begin{definition}
A dynamical system $(X,f)$, where $X$ is a Hausdorff space, is said to be Hausdorff sensitive if there exists a finite open cover $\mathcal{U}$ such that for any nonempty open set $V$ there exist $x,y \in V$, $x\neq y$, and $k 
\in \mathbb{N}$ such that $\{f^k(x),f^k(y)\} \not\subseteq U$ for all $U \in \mathcal{U}$.
\end{definition}

\begin{theorem}\label{theoremDOHsdoic}
Let $(X,f)$ be a system with a transitive point $x$, where $X$ is a regular Hausdorff space (i.e. $T_3$). If there exists $y \in X$ with $(x,y) \notin \EqP(X,f)$ then $(X,f)$ is Hausdorff sensitive.
\end{theorem}

\begin{proof}
Let $x$ and $y$ be as in the statement. Therefore
\begin{equation}
   \exists O \in \mathcal{N}_y : \forall U \in \mathcal{N}_x \, \forall V \in \mathcal{N}_y \, \exists n \in \mathbb{N} : f^n(U) \cap V \neq \emptyset \text{ and } f^n(U) \not\subseteq O.
\end{equation}
Let $V_1$ and $V_2$ be open neighbourhoods of $y$ such that $\overline{V_1} \subseteq O$ and $\overline{V_2} \subseteq V_1$; these exist as $X$ is regular.
Then $\mathcal{U}\coloneqq \{V_1, X \setminus \overline{V_2}\}$ is a finite open cover. Now let $U$ be an arbitrary nonempty open set. Let $n \in \mathbb{N}$ be such that $W=f^{-n}(U) \ni x$. Take $m>n$ such that $f^m(W) \cap V_2 \neq \emptyset$ and $f^m(W) \not \subseteq O$; such an $m$ exists by Lemma \ref{RemarkInf1}. Then $f^{m-n}(U) \cap V_2 \neq \emptyset$ and $f^{m-n}(U) \not \subseteq O$. In particular there exists $a,b \in U$ such that $f^{m-n}(a) \notin O$ and $f^{m-n}(b) \in V_2$. Then $\{f^{m-n}(a), f^{m-n}(b) \} \cap V_1= \{f^{m-n}(b)\}$ and $\{f^{m-n}(a), f^{m-n}(b) \} \cap X \setminus \overline{V_2}= \{f^{m-n}(a)\}$. \end{proof}

\begin{corollary}\label{GenA-YDichI}(Generalised Auslander-Yorke Dichotomy I) 
Let $X$ be a $T_3$ space. A minimal system $(X,f)$ is either topologically equicontinuous or Hausdorff sensitive.
\end{corollary}

\begin{proof}
Suppose it is not equicontinuous. Then there exists $x,y\in X$ with $(x,y) \notin \EqP(X,f)$. Since $x$ is a transitive point the result follows from Theorem \ref{theoremDOHsdoic}.
\end{proof}

We end this section with the following question. 


\begin{question}\label{QuestionDoesExistsTransEP}
Does there exist a transitive system $(X,f)$, where $X$ is a Hausdorff space, with a non-trivial equicontinuity pair $(x,y)$ but where $x$ is not a topological equicontinuity point?
\end{question}


The following result may help make some headway with Question \ref{QuestionDoesExistsTransEP}. 

\begin{proposition}\label{PE}
Suppose $(X,f)$ is a transitive dynamical system where $X$ is an infinite Hausdorff space. If $x \in X$ is an eventually periodic point then $(x,y)\notin \EqP(X,f)$ for any $y \in X$.
\end{proposition}

\begin{proof}
Write $\Orb(x)=\{x,f(x),\ldots, f^l(x)\}$. Suppose $(x,y) \in \EqP(X,f)$. Then by Corollary \ref{Cor1.Eq} it follows that $y \in \omega(x)$; as $x$ is eventually periodic this means $y \in \Orb(x)$ and $y$ is periodic. Write $y=f^m(x)$ and let $n$ be the period of $y$ (so $n \leq l$). Let $z \in X \setminus \Orb(x)$; for each $i \in \{0,\ldots, n-1\}$ let $W_i \in \mathcal{N}_z$ and $O_i\in \mathcal{N}_{f^i(y)}$ be such that $W_i \cap O_i = \emptyset$. Now let 
\[O \coloneqq \bigcap_{i=0} ^{n-1} f^{-i}(O_i),\]
and
\[W \coloneqq \bigcap_{i=0} ^{n-1} W_i.\]
Thus $W \in \mathcal{N}_z$ and $O \in \mathcal{N}_y$. Now let $U \in \mathcal{N}_x$ and $V \in \mathcal{N}_y$ satisfy the equicontinuity condition for $x$, $y$ and $O$. Notice that $f^i(O) \cap W = \emptyset$ for all $i \in \{0, \ldots, n-1\}$. Since $f^m(U) \cap V \neq \emptyset$ we have $f^m(U) \subseteq O$. Furthermore, $f^{m+an}(U) \cap V \neq \emptyset$ for all $a \in {\mathbb{N}_0}$, hence $f^{m+an}(U) \subseteq O$. It follows that $f^k(U) \cap W =\emptyset$ for all $k \geq m$, this contradicts Lemma \ref{lemmaEq0.1}.
\end{proof}

\section{Even continuity}\label{SectionEvenCty}

Even continuity, as defined by Kelley \cite[p. 234]{Kelley}, is a weaker concept than that of topological equicontinuity. If $X$ and $Y$ are topological spaces we say a collection of maps $\mathscr{F}$ from $X$ to $Y$ is \textit{evenly continuous} at an ordered pair $(x,y) \in X \times Y$ if for any $O \in \mathcal{N}_y$ there exist neighbourhoods $U \in \mathcal{N}_x$ and $V \in \mathcal{N}_y$ such that, for any $f \in \mathscr{F}$, if $f(x) \in V$ then $f(U) \subseteq O$; when this is the case we refer to $(x,y)$ as an \textit{even continuity pair}. We say $\mathscr{F}$ is \textit{evenly continuous} at a point $x \in X$ if it is evenly continuous at $(x,y)$ for all $y\in Y$. We say the collection is \textit{evenly continuous} if it is evenly continuous at every $x \in X$. We remark that when $Y$ is a compact Hausdorff space the notions of topological equicontinuity, even continuity and equicontinuity coincide (see \cite[Theorem 7.23]{Kelley}). Finally, we observe that if a family is evenly continuous (resp. topological equicontinuous) then each member of that family is necessarily continuous \cite[pp. ~162]{Engelking}.

Given a dynamical system $(X,f)$, we denote the collection of even continuity pairs and the collection of even continuity points by $\EvP(X,f) \subseteq X \times X$ and $\Ev(X,f) \subseteq X$ respectively. Note that in this case, if we consider the above definitions, we have $Y=X$ and $\mathscr{F}=\{f^n \mid n \in \mathbb{N}\}$. By definition it follows that $(X,f)$ is evenly continuous precisely when $\EvP(X,f) = X \times X$. For $(x,y) \in \EvP(X,f)$, we refer to the condition
\begin{equation}\label{eqnEvCondition}
\forall O\in \mathcal{N}_y \, \exists U \in \mathcal{N}_x \, \exists V \in \mathcal{N}_y : \forall n \in \mathbb{N}, f^n(x) \in V \implies f^n(U) \subseteq O,
\end{equation}
as the \textit{even continuity condition} for $x$ and $y$. We say that $U$ and $V$, as in Equation \ref{eqnEvCondition}, satisfy the even continuity condition for $x$, $y$ and $O$.


\begin{remark}\label{Remark5.0}
Clearly every equicontinuity pair is an even continuity pair.
\end{remark}

As pointed out by others (e.g. \cite{Royden}), the converse to Remark \ref{Remark5.0} is not true in general. The following example demonstrates this.

\begin{example}
For each $n \in \mathbb{N}$, let $X_n$ be the finite word $10^n$ and take $x=X_1X_2X_3\ldots$. For each $n \in \mathbb{N}$, let $z_n=X_1X_2\ldots X_n0^\infty$. Let $y=0^\infty$. Take
\[Y \coloneqq \{0^m z_n, 0^m x, 0^\infty \mid n,m\in \mathbb{N}\},\]
and let
\[X\coloneqq \overline{\{\sigma^k(y) \mid y \in Y, k \in {\mathbb{N}_0}\}}.\]
where the closure is taken with regard to the full shift $\Sigma^{\mathbb{N}_0}$.

Note that,
\[\omega(x)=\{0^\infty, 0^n10^\infty \mid n \in {\mathbb{N}_0}\},\]
and for each $i \in \mathbb{N}$,
\[\omega(z_i)=\{0^\infty \}.\]


Considering the shift system $(X,\sigma)$, it is easy to see that $(x,0^\infty)$ is a non-trivial even continuity pair in $(X,\sigma)$ (i.e. it is an even continuity pair and $0^\infty \in \omega(x)$). Furthermore, it is not an equicontinuity pair; arbitrarily close to $x$ are points that map onto $0^\infty$, which is a fixed point, but $x$ itself is not pre-periodic. To show this explicitly, take $O=[0] \in \mathbb{N}_{0^\infty}$. Picking $U \in \mathcal{N}_x$, there exists $N \in \mathbb{N}$ such that $f^k(U) \ni 0^\infty$ for all $k >N$; in particular, for any $V \in \mathcal{N}_{0^\infty}$, $f^k(U) \cap V \neq \emptyset$ for all $k\geq N$. But there exists $k\geq N$ such that $f^k(x)\in [1]$, hence $(x,0^\infty) \notin \EqP(X,f)$.
\end{example}


\begin{proposition}\label{prop5.0}
Let $(X,f)$ be a dynamical system where $X$ Hausdorff space. Let $x,y \in X$ and suppose $y \notin \omega(x)$. Then $(x,y) \in \EvP(X,f)$.
\end{proposition}
\begin{proof}
Let $O \in \mathcal{N}_y$ be given. Since $y \notin \omega(x)$ there exist $V \in \mathcal{N}_y$ and $N \in \mathbb{N}$ such that $f^n(x) \notin V$ for all $n> N$. 

By Corollary \ref{corFiniteTopEqCondition}, there exist $U^\prime$ and $V^\prime$ such that, for any $k \in \{1,\ldots, N\}$, if $f^k(U^\prime) \cap V^\prime \neq \emptyset$ then $f^k(U^\prime) \subseteq O$; without loss of generality $U^\prime \subseteq U$ and $V^\prime \subseteq V \cap O$. In particular this means that, for all $k \in \{1,\ldots, N\}$, if $f^k(x) \in V^\prime$ then $f^k(U) \subseteq O$. Then, since $f^n(x) \notin V^\prime$ for all $n>N$, $U^\prime$ and $V^\prime$ satisfy the even continuity condition for $x$, $y$ and $O$. As $O \in \mathcal{N}_y$ was picked arbitrarily the result follows.
\end{proof}



\begin{remark}
If $X$ is a Hausdorff space, putting together propositions \ref{prop1.0} and \ref{prop5.0}, we have, for a pair $x,y \in X$, the following: 
\begin{itemize}
\item If $y \notin \omega(x)$ then $(x,y) \in \EvP(X,f)$.
\item If $y \notin \Omega(x)$ then $(x,y) \in \EqP(X,f)$. 
\end{itemize} 
\end{remark}


\begin{definition}
If $(X,f)$ is a dynamical system, where $X$ is a Hausdorff space. We say $(x,y) \in X \times X$ is a trivial even continuity pair if $y \notin \omega(x)$.
\end{definition}

\begin{remark} Proposition \ref{prop5.0} tells us that a trivial even continuity pair is indeed an even continuity pair. We emphasise that, by definition, if $(x,y) \in \EvP(X,f)$ then either they are a trivial even continuity pair or $y \in \omega(x)$. Finally, it is worth observing that, by Lemma \ref{lemma1.0} and Remark \ref{Remark5.0}, a non-trivial equicontinuity pair is also a non-trivial even continuity pair.
\end{remark}

The statement $(x,y) \notin \EvP(X,f)$, for $x,y \in X$, means precisely
\begin{equation}\label{eqnEvSplit}
    \exists O \in \mathcal{N}_y : \forall U \in \mathcal{N}_x \, \forall V \in \mathcal{N}_y \, \exists n \in \mathbb{N} : f^n(x)\in V \text{ and } f^n(U) \not\subseteq O.
\end{equation}
We shall refer to a neighbourhood such as $O$ in equation \ref{eqnEvSplit} as an \textit{even-splitting neighbourhood} of $y$ with regard to $x$. It is straightforward to see that every even-splitting neighbourhood of $y$ with regard to $x$ is also a splitting neighbourhood of $y$ with regard to $x$. Notice that, by Proposition \ref{prop5.0}, if $(x,y) \notin \EvP(X,f)$ then $y \in \omega(x)$.

The proof of lemmas \ref{lemmaEvPOpenAt} and \ref{RemarkInf3} are very similar to that of lemmas \ref{lemmaEqPOpenAt} and \ref{RemarkInf1} respectively and are thereby omitted.

\begin{lemma}\label{lemmaEvPOpenAt}
Let $(X,f)$ be a dynamical system, where $X$ is a Hausdorff space. If $(x,y) \in \EvP(X,f)$, $f$ is open at $y$ and there is a neighbourhood base for $y$, $\mathcal{B}_y\subseteq \mathcal{N}_y$, such that $f^{-1}(f(O \cap \Orb(x)))=O\cap \Orb(x)$ for all $O \in \mathcal{B}_y$, then $(x,f(y)) \in \EvP(X,f)$.
\end{lemma}



\begin{corollary}
Let $(X,f)$ be a dynamical system, where $X$ is a Hausdorff space. If $f$ is a homeomorphism and $(x,y) \in \EvP(X,f)$ then $(x,f(y)) \in \EqP(X,f)$.
\end{corollary}

\begin{proof}
Immediate from Lemma \ref{lemmaEvPOpenAt}.
\end{proof}

\begin{lemma}\label{RemarkInf3} Let $(X,f)$ be a dynamical system, where $X$ is Hausdorff space. Suppose $(x,y)\notin \EvP(X,f)$ and let $O$ be an even-splitting neighbourhood of $y$ with regard to $x$. Then, for any pair of neighbourhoods $U$ and $V$ of $x$ and $y$ respectively, the set of natural numbers $n$ for which $f^n(x) \in V$ and $f^n(U) \not\subseteq O$ is infinite.
\end{lemma}


\begin{lemma}\label{OrbNonEvPair}
Let $(X,f)$ be a dynamical system, where $X$ is Hausdorff space. Let $x,y \in X$. If $(x,y)\notin \EvP(X,f)$ and $O$ is an even-splitting neighbourhood of $y$ with regard to $x$ then, for any $n \in \mathbb{N}$, $(f^n(x), y)\notin \EvP(X,f)$ and $O$ is an even-splitting neighbourhood of $y$ with regard to $f^n(x)$.
\end{lemma}

\begin{proof}
Let $U \in \mathcal{N}_{f^n(x)}$ and $V \in \mathcal{N}_y$. Then $W=f^{-n}(U) \in \mathcal{N}_x$. 
By Lemma \ref{RemarkInf3} the set
\[A= \{k \in \mathbb{N} \mid f^k(x) \in V  \text{ and } f^k(W) \not\subseteq O \},\]
is infinite. Taking $m>n$ with $m \in A$ gives the result.
\end{proof}

\begin{remark}
We emphasise the contrapositive of Lemma \ref{OrbNonEvPair}: 
Let $(X,f)$ be a dynamical system, where $X$ is Hausdorff space. Suppose $(x,y) \in \EvP(X,f)$ and $x \in \Orb(z)$. Then $(z, y)\in \EvP(X,f)$.
\end{remark}

\begin{proposition}
Let $X$ be a Hausdorff space. If $(X,f)$ is a dynamical system and there exists a point $x \in X$ with no even continuity partners, then $(X,f)$ has no equicontinuity pairs. Furthermore, such a point is a transitive point.
\end{proposition}

\begin{proof} Note first that such a point is clearly a transitive point.

Let $x\in X$ be a point with no even continuity partners. Let $y,z \in X$ be picked arbitrarily. Let $O$ be an even-splitting neighbourhood of $z$ with regard to $x$. Let $U \in \mathcal{N}_y$ and $V \in \mathcal{N}_z$. As $x$ is transitive there exists $n \in \mathbb{N}$ such that $f^n(x) \in U$. By Lemma \ref{OrbNonEvPair}, $(f^n(x), z)$ is not an even continuity pair and $O$ is an even-splitting neighbourhood of $z$ with regard to $f^n(x)$. It follows that there exists $m \in \mathbb{N}$ such that $f^m(U)\cap V \neq \emptyset$ and $f^m(U) \not\subseteq O$; hence $(y,z) \notin \EqP(X,f)$.
\end{proof}



At the end of the previous section, we asked, in Question \ref{QuestionDoesExistsTransEP}, whether there exists a transitive system with an equicontinuity pair but no point of equicontinuity. We now answer, in the positive, an analogous question with regard to even continuity pairs.

\begin{theorem}\label{ThmExistsTransEvP}
There exists a transitive system $(X,f)$ with a non-trivial even continuity pair but no point of even continuity: Furthermore, there is such a system which is additionally Auslander-Yorke chaotic, densely and strongly Li-Yorke chaotic, but not Devaney chaotic, whilst having no equicontinuity pairs. 
\end{theorem}

Due to the length and technical nature of the proof of Theorem \ref{ThmExistsTransEvP} we leave it until the end of the paper (Section \ref{SectionProofOfExample}).

\begin{remark}
Devaney \cite[pp. 50]{Devaney} defined chaos as a topologically transitive, sensitive system with a dense set of periodic points. This last property means that, ``in the midst of random behaviour, we nevertheless have an element of regularity.'' The construction in the proof of Theorem \ref{ThmExistsTransEvP} shows that a system which is, in some sense, extremely chaotic (it is not only sensitive but expansive, whilst having only two periodic points) can still exhibit some element of regularity: the even continuity pair $(x,\texttt{0}^\infty)$ provides some regularity associated with $x$. When $x$ moves close to $\texttt{0}^\infty$, everything from a certain neighbourhood of $x$ also moves close to $\texttt{0}^\infty$.
\end{remark}

The following corollary is an immediate consequence of Theorem \ref{ThmExistsTransEvP}.

\begin{corollary}
The notions of equicontinuity pair and even continuity pair, in general, remain distinct for transitive dynamical systems.
\end{corollary}


The last result in this section is a variation on Proposition \ref{PE}; it gives us some information about the types of pairs which cannot be even continuity pairs in transitive systems.

\begin{proposition}\label{EvenNoPeriodicPt}
Suppose $(X,f)$ is a transitive dynamical system where $X$ is an infinite Hausdorff space. If $x \in X$ is an eventually periodic point then $(x,y)$ is not a non-trivial even continuity pair for any $y \in X$.
\end{proposition}

The proof of Proposition \ref{EvenNoPeriodicPt} is very similar to that of Proposition \ref{PE} and is thereby omitted.





\section{Equicontinuity, transitivity and splitting}\label{SectionSplitting}

A subset $N=\{n_1,n_2,n_3, \ldots\} \subseteq \mathbb{N}$, where $n_1 <n_2 <n_3 \ldots$ is said to be {\em syndetic} if there exists $l \in \mathbb{N}$ such that $n_{i+1}-n_i \leq l$; such an $l$ is called a \textit{bound of the gaps}. A subset is called {\em thick} if it contains arbitrarily long strings without gaps. A subset is called cofinite if its complement is finite. Using this, a dynamical system $(X,f)$ is said to be 
\begin{enumerate}
\item Syndetically (resp. thickly) transitive if $N(U,V)$ is syndetic (resp. thick) for any nonempty open $U$ and $V$.
\item Syndetically (resp. thickly / resp. cofinitely) sensitive if there exist a compatible uniformity $\mathscr{D}$ and a symmetric $D \in \mathscr{D}$ such that, for any nonempty open $U \subseteq X$, the set $N_D(U)$ is syndetic (resp. thick / resp. cofinite).
\item Strong mixing if $N(U,V)$ is cofinite for any nonempty open $U$ and $V$.
\end{enumerate}



In this section we investigate the link between topological equicontinuity, transitivity and sensitivity. 
Trivially, if a dynamical system has an equicontinuity point then it is not sensitive. If we restrict our attention to compact metric systems, adding the condition of transitivity is enough to give a partial converse; a transitive map with no equicontinuity points is sensitive \cite{AkinAuslanderBerg}.
The proof provided by Akin \textit{et al} does not rely on the space being metrizable; with only minor adjustments the result generalises to give the following.

\begin{theorem}\textup{\cite{AkinAuslanderBerg}}\label{2.1} Let $(X,f)$ be a dynamical system, where $X$ is a compact Hausdorff space. If there exists a transitive point and $\Eq(X, f)=\emptyset$ then $(X,f)$ is sensitive.
\end{theorem}

If $X$ is a compact metric space, and $(X,f)$ a transitive dynamical system, then there exists a transitive point (since $X$ is non-meagre and has a countable $\pi$-base
).
By Theorems \ref{2.1} and \ref{ChT} it follows that, for a compact metric system, no equicontinuity pairs implies both transitivity and sensitive dependence on initial conditions.

\begin{corollary}\label{CorIfEPEmptyThenTransAndSens}
Let $X$ be a compact Hausdorff space that is non-meagre and which yields a countable $\pi$-base. If $\EqP(X,f)=\emptyset$ then the system is both transitive and sensitive. 
\end{corollary}

\begin{proof}
Apply Theorems \ref{2.1} and \ref{ChT}.
\end{proof}



\begin{proposition}
Let $X$ be a Hausdorff space and $(X,f)$ a dynamical system. If $X$ is nonmeagre with a countable $\pi$-base, then $\EqP(X,f) = \emptyset$ if and only if there exists a transitive point $x \in X$ with no equicontinuity partners.
\end{proposition}

\begin{proof}
Assume the latter and let $x$ be such a transitive point. Suppose that $(a,b) \in \EqP(X,f)$, for some $a,b \in X$. Then $(a,b)$ is a non-trivial equicontinuity pair by Theorem \ref{ChT}. As $x$ is a transitive point $a \in \omega(x)$. It follows from Corollary \ref{CorNonTrivEqPairIsTrans} that $(x,b) \in \EqP(X,f)$, a contradiction.

Now suppose the former. By Corollary \ref{CorIfEPEmptyThenTransAndSens} the system is transitive, which entails the existence of a transitive point as $X$ is nonmeagre with a countable $\pi$-base.
\end{proof}




We now turn our attention to examining sufficient conditions for $\EqP(X,f)=\emptyset$. One obvious such condition is the following.


\begin{proposition}\label{prop} Let $(X,f)$ be a dynamical system, where $X$ is a Tychonoff space. Suppose there exists a compatible uniformity $\mathscr{D}$ and a symmetric $D \in \mathscr{D}$ such that for any nonempty open sets $U$ and $V$, $N(U,V) \cap N_D(U) \neq \emptyset$, then there are no equicontinuity pairs.
\end{proposition}


Note that, when the hypothesis of this proposition occurs, it is equivalent to being able to move the existential quantifier to the front of the statement stating $\EqP(X,f)=\emptyset$. To be clear, $\EqP(X,f)=\emptyset$ means, 
\[ \forall x,y \in X \, \exists D \in \mathscr{D} : \forall U\in \mathcal{N}_x \, \forall V\in \mathcal{N}_y,  N(U,V) \cap N_D(U) \neq \emptyset.\]
whilst the hypothesis states,
\[ \exists D \in \mathscr{D} : \forall x,y \in X \, \forall U\in \mathcal{N}_x \, \forall V\in \mathcal{N}_y, \, N(U,V) \cap N_D(U) \neq \emptyset.\]

For any pair of sets $U,V \subseteq X$, we define $N_D(U,V) \coloneqq N(U,V) \cap N_D(U)$; if $X$ is a metric space and $\delta>0$ we similarly define $N_\delta(U,V) \coloneqq N(U,V) \cap N_\delta(U)$. Such a set is extremely relevant in an applied setting, where small rounding errors mean that a different point than the one intended might be being tracked. This set tells us precisely when $U$ meets $V$ whilst also expanding to at least diameter $\delta$. 
The importance of such a set leads us to give the following definition.

\begin{definition}
Let $(X,f)$ be a dynamical system, where $X$ is a Tychonoff space. We say that $(X,f)$ experiences splitting if there is a compatible uniformity $\mathscr{D}$ and a symmetric $D \in \mathscr{D}$ such that for any pair of nonempty open sets $U$ and $V$ we have $N_D(U,V) \neq \emptyset$. Such a $D$ is called a splitting entourage for $(X,f)$.
\end{definition}
In similar fashion, if $X$ is a metric space we say the system $(X,f)$ has \textit{splitting} if there exists $\delta>0$ such that for any pair of nonempty open sets $U$ and $V$ we have $N_\delta(U,V) \neq \emptyset$. Thus a system has splitting when every nonempty open set `hits' every other such set whilst simultaneously being pulled apart to diameter at least $\delta$. Proposition \ref{prop} then states that any splitting system has no equicontinuity pairs. The following lemma is analogous to several previously stated.

\begin{lemma}\label{RemarkInf2} If $(X,f)$ is a Tychonoff system with splitting, with splitting entourage $D$, then for any nonempty open pair $U$ and $V$, $N_D(U,V)$ is infinite.
\end{lemma}

\begin{proof}
Suppose $N_D(U,V)$ is finite. Since $(X,f)$ has splitting, with splitting entourage $D$, $N_D(U,V) \neq \emptyset$. Let $k \in \mathbb{N}$ be the greatest element of $N_D(U,V)$. Let $W \subseteq U \cap f^{-k}(V)$ be open such that, for any $i \in \{1, \ldots, k\}$ and any $x,y \in W$, $(f^i(x),f^i(y)) \in D$. As $N_D(W,V) \neq \emptyset$ and $W \subseteq U$ we have a contradiction and the result follows.
\end{proof}

\begin{corollary}
Let $X$ be a Tychonoff space with at least two points. If $(X,f)$ is weakly mixing, then $(X,f)$ experiences splitting. 
\end{corollary}

\begin{proof} 
Suppose $(X,f)$ exhibits weak mixing. Let $\mathscr{D}$ be a compatible uniformity for $X$ and $E \in \mathscr{D}$ be a symmetric entourage such that, for any $x \in X$, we have $E[x] \neq X$. Let $D \in \mathscr{D}$ be symmetric such that $2D \subseteq E$. Let $U$ and $V$ be nonempty open sets. Let $x \in V$ and pick $y \in X$ such that $(x,y) \notin E$. By weak mixing, there exists $n \in \mathbb{N}$ such that $f^n(U) \cap \left(D[x]\cap V\right) \neq \emptyset$ and $f^n(U) \cap D[y] \neq \emptyset$. Let $u \in f^n(U) \cap \left(D[x]\cap V\right)$ and $u^\prime \in f^n(U) \cap D[y]$; by symmetry $(x,u) \in D$ and $(u^\prime,y) \in D$. If $(u, u^\prime) \in D$ then $(x,y) \in 2D \subseteq E$, a contradiction.
\end{proof}
We remark that if $(X,f)$ is topologically exact or has strong mixing then it has weak mixing; this means each of these properties are also sufficient for a system to have splitting.\footnote{The system $(X,f)$ is topologically exact if, for any nonempty open set $U$ there exists $n \in {\mathbb{N}_0}$ such that $f^n(U)=X$.}

Clearly we also have the following result.

\begin{proposition} Let $(X,f)$ be a dynamical system, with $X$ a Tychonoff space. Let $P$ and $Q$ be properties of subsets $\mathbb{N}$ such that if $A$ and $B$, subsets of $\mathbb{N}$, have $P$ and $Q$ respectively, then $A \cap B \neq \emptyset$. Then if $(X,f)$ is $P$-transitive (by which we mean for any pair of nonempty open sets $U$ and $V$, $N(U,V)$ has property $Q$) and $Q$-sensitive (by which we mean there exist a compatible uniformity $\mathscr{D}$ and a symmetric entourage $D \in \mathscr{D}$ such that for any nonempty open set $U$, $N_D(U)$ has property $P$), then it experiences splitting. 
\end{proposition}

For example, if $(X,f)$ is syndetically transitive and thickly sensitive it follows that it has splitting. Also, since transitivity implies $N(U,V)$ is infinite for any nonempty open pair $U$ and $V$,
we have that a transitive system which is cofinitely sensitive has splitting; in particular any transitive map on $[0,1]$ has splitting.\footnote{Any such map is cofinitely sensitive (see \cite{Moothathu}).}

It turns out that any Devaney chaotic system on a compact space has splitting, and consequently has no equicontinuity pairs. We will see that this follows as a corollary to Theorem \ref{thm}.

\begin{theorem}\label{thm} Let $(X,f)$ a syndetically transitive dynamical system, where $X$ is a compact Hausdorff space. If there are two distinct minimal sets then there exists a symmetric entourage $D \in \mathscr{D}$ such that for any nonempty open pair $U$ and $V$, $N_D(U,V)$ is syndetic; i.e. the system experiences \textit{syndetic splitting}. 
\end{theorem}

(NB. The proof below mimics Moothathu's \cite[Theorem 1]{Moothathu} proof that a non-minimal syndetically transitive system has syndetic sensitivity for metric systems.)

\begin{proof}
Let $M_1$ and $M_2$ be distinct minimal sets; it follows that $M_1 \cap M_2 =\emptyset$. Let $x \in M_1$ and $y \in M_2$; so $\overline{\Orb(x)}=M_1$ and $\overline{\Orb(y)}=M_2$. Let $D \in \mathscr{D}$ be symmetric such that, for any $z_1 \in M_1$ and any $z_2 \in M_2$, $(z_1,z_2) \notin 8D$. Now let $U$ and $V$ be nonempty open sets and take $z \in V$; without loss of generality $V \subseteq D[z]$. Suppose there is $p \in M_1$ and $q \in M_2$ such that $(p,z) \in 4D$ and $(z,q) \in 4D$; then $(p,q) \in 8D$, contradicting our choice of $D$. Without loss of generality we may thereby assume $(p,z) \notin 4D$ for any $p \in M_1$. Let $l_1$ be a bound of the gaps for $N(U,V)$. 
Let $W \ni x$ be open such that if $w \in W$ then $(f^i(w),f^i(x)) \in D$ for all $i \in \{0,1,\ldots, l_1\}$; $W$ exists by continuity. By construction, for any $w \in W$, any $v \in V$ and any $i \in \{0,1,\ldots, l_1\}$ we have $(f^i(w), v) \notin 2D$. Let $l_2$ be a bound of the gaps for $N(U,W)$. 
It can now be verified that $N(U,V)\cap N(U,W)$ is itself syndetic, with $l_1+l_2$ a bound of the gaps. 
Since $N_D(U,V) \supseteq N(U,V)\cap N(U,W)$ the result follows.
\end{proof}


The following corollaries follow from Theorem \ref{thm} and Proposition \ref{prop}.

\begin{corollary} Let $(X,f)$ be a syndetically transitive dynamical system, where $X$ is a compact Hausdorff space. If there are two distinct minimal sets then there are no equicontinuity pairs.
\end{corollary}

\begin{corollary}\label{cor2} Let $(X,f)$ be a non-minimal transitive system with a dense set of minimal points, where $X$ is a compact Hausdorff space. Then 
the system is syndetically splitting.
\end{corollary}

\begin{proof} Moothathu \cite{Moothathu} shows that a transitive system with a dense set of minimal points is syndetically transitive. If the system is non-minimal but the set of minimal points is dense, there exist multiple minimal sets.
\end{proof}

\begin{corollary} Let $(X,f)$ be a Devaney chaotic dynamical system where $X$ is a compact Hausdorff space. Then $(X,f)$ experiences syndetic splitting.
\end{corollary}

\begin{proof} This follows from Corollary \ref{cor2}.
\end{proof}

\begin{corollary}
Let $(X,f)$ exhibit shadowing and chain transitivity, where $X$ is a compact Hausdorff space. If there are two distinct minimal sets then the system has syndetic splitting.
\end{corollary}

\begin{proof}
Li \cite{Li} shows that a non-minimal compact metric system with shadowing and chain transitivity is syndetically transitive; this result generalises easily to compact Hausdorff systems. The result follows from Theorem \ref{thm}.
\end{proof}


\begin{question}\label{QuestionSplittingVsAYChaos}
If $X$ is a Tychonoff space, is splitting distinct from Auslander-Yorke chaos? 
\end{question}
We asked previously (Question \ref{QuestionDoesExistsTransEP}), whether or not a transitive system can have an equicontinuity pair $(x,y)$ without the system being equicontinuous at $x$. A more restrictive question is the following: Is it possible for a transitive point to have an equicontinuity partner but not be an equicontinuity point? 
This itself is related to Question \ref{QuestionSplittingVsAYChaos}. Indeed, if there exists a compact Hausdorff system $(X,f)$, with a transitive point $x \notin \Eq(X,f)$ and a point $y \in X$ with $(x,y) \in \EqP(X,f)$, then it would follow that splitting is not equivalent to Auslander-Yorke chaos; such a system would be both transitive and, since there would be no equicontinuity points by Theorem \ref{thmTopEqPtsCoincideWithTransPts}, sensitive (Theorem \ref{2.1}). However, for any entourage $D \in \mathscr{D}$, there would exist neighbourhoods $U \in \mathcal{N}_x$ and $V \in \mathcal{N}_y$ such that $N_D(U,V)=\emptyset$, hence the system would not have splitting.




    
    

\section{Eventual sensitivity}\label{SectEventualSensitivity}
The following definition was motivated by the following thought: sensitive dependence on initial conditions means that, no matter where you start, there are two points arbitrarily close to each other and to that starting location which will move far apart as time progresses; a universal `far'. Clearly this is extremely relevant in an applied setting; rounding errors mean a computer will not, generally, track true orbits. But what if every point moves arbitrarily close to another point that it will then move away from? What if a computer starts with a true orbit and tracks it accurately, but then the point moves close to another point which will end up going in completely the other direction? - these two points may be so close together that the computer cannot differentiate between them; it may start tracking the wrong orbit and give an extremely inaccurate prediction of the future.


\begin{definition} 
We say a metric dynamical system $(X,f)$ is eventually sensitive if there exists $\delta>0$ such that for any $x \in X$ and any $\epsilon>0$ there exists $n,k \in \mathbb{N}$ and $y \in B_\epsilon(f^n(x))$ such that $d(f^{n+k}(x), f^k(y)) \geq \delta$. We refer to such a $\delta$ as an eventual-sensitivity constant.

If $X$ is a compact Hausdorff space, we say that $(X,f)$ is eventually sensitive if there exists $D \in \mathscr{D}$ such that for any $x \in X$ and any $E \in \mathscr{D}$ there exists $n,k \in \mathbb{N}$ and $y \in E[f^n(x)]$ such that $(f^{n+k}(x),f^k(y)) \notin D$. We refer to such a $D$ as an eventual-sensitivity entourage.

\end{definition}

Clearly a system which is sensitive is also eventually sensitive; just take $n=0$ in the above definition. The variable $n$ is something that needs to be taken into account in an applied setting (and clearly it may depend on one's starting point); if the least such $n$ is large, then the computer may provide an accurate model of the reasonably distant future. However, if the least such $n$ is small, or $0$ as in the case of sensitivity, the orbit the computer is attempting to track may quickly diverge from what the computer predicts. The example below is an example of an eventually sensitive but non-sensitive system.

\begin{example}\label{examplES}
Let $X=[0,1]$. Define a map $f\colon X \to X$ by
\[f(x)=\left\{\begin{array}{lll}
2x & if & x \in [0,\frac{1}{4}],
\\1-2x & if & x \in [\frac{1}{4},\frac{1}{2}] ,
\\ \frac{10}{3}x-\frac{5}{3} & if & x \in [\frac{1}{2},\frac{3}{5}] ,
\\ \frac{1}{3} & if & x \in [\frac{3}{5},\frac{4}{5}] ,
\\ \frac{10}{3}x-\frac{7}{3} & if & x \in [\frac{4}{5},1].
\end{array}\right.\]
Then $f\colon X \to X$ is a continuous surjection which is eventually sensitive but not sensitive.
\end{example}

The system in the example above (\ref{examplES}) is not sensitive; the point $\frac{3}{4}$ has a neighbourhood on which the map is constant. However it is eventually sensitive; to see this notice that every point in $[0,1)$ has its $\omega$-limit set in $[0,\frac{1}{2}]$, where the map is simply a copy of the tent map, which is sensitive (indeed, it is cofinitely so). Finally, $1$ is a fixed point $f(1)=1$, which is of a fixed distance $\frac{1}{2}$ from the interval $[0, \frac{1}{2}]$. Figure \ref{figureES} below shows this map.

\begin{figure}[h]
\centering
\begin{tikzpicture}[scale=4]
\datavisualization [school book axes,
                    visualize as smooth line,
                    y axis={label},
                    x axis={label} ]

data [format=function] {
      var x : interval [0:0.25] samples 30;
      func y = 2* \value x;
      }
data [format=function] {
      var x : interval [0.25:0.5] samples 30;
      func y = 1+(-2) * \value x;
      }
data [format=function] {
      var x : interval [0.5:0.6] samples 30;
      func y = -(5/3)+(10/3) * \value x;
      }
data [format=function] {
      var x : interval [0.6:0.8] samples 30;
      func y = 1/3;
      }
data [format=function] {
      var x : interval [0.8:1.0] samples 30;
      func y = -(7/3)+(10/3) * \value x;
      };
\draw (-0.07,1) -- (0.07,1);      
\node at (-0.1,1) {\footnotesize 1};
\end{tikzpicture}
\caption{A non-sensitive, eventually-sensitive system}
\label{figureES}
\end{figure}
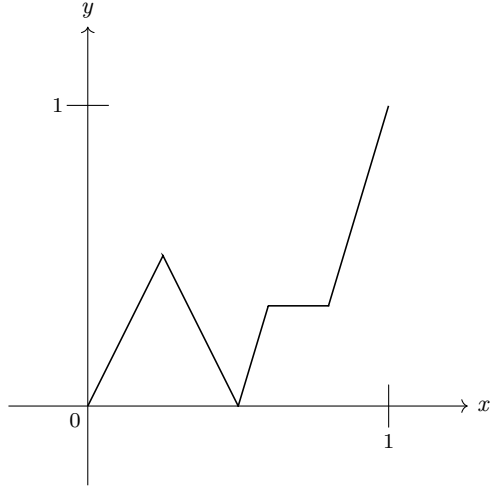

For transitive dynamical systems we prove the following dichotomy.

\begin{theorem}(Generalised Auslander-Yorke Dichotomy II) \label{ESD1} Let $X$ be a compact Hausdorff space. A transitive dynamical system $(X,f)$ is either equicontinuous or eventually sensitive. Specifically, it is eventually sensitive if and only if it is not equicontinuous.
\end{theorem}

\begin{proof} 
Suppose first that the system is not equicontinuous. Suppose the system has a dense set of minimal points. If the system is minimal then it is sensitive (see \cite[Corollary 2]{AuslanderYorke} or Corollary \ref{GenA-YDichI}) and the result follows. If it is non-minimal then it is sensitive (see \cite[Thereom 2.5]{AkinAuslanderBerg}) and therefore eventually sensitive. Now suppose the set of minimal points $M$ is not dense in $X$. Let $q \in X$ and $D \in \mathscr{D}$ be symmetric such that $3D[q] \cap \overline{M} =\emptyset$. Let $z \in X$ be picked arbitrarily and let $E \in\mathscr{D}$ be given; without loss of generality $E \subset D$. Let $m \in \omega(z)$ be minimal. Then there exists $n \in \mathbb{N}$ such that $m \in E[f^n(z)]$. By transitivity, there exists $k \in \mathbb{N}$ such that $ f^k\left( E[f^n(z)]\right) \cap D[q] \neq \emptyset$. Let $y \in E[f^n(z)]$ be such that $f^k(y) \in D[q]$. Then $(f^k(y),f^k(m))\notin 2D$ as $f^k(m)$ is minimal. Then either $(f^{n+k}(z),f^k(y))\notin D$ or $(f^{n+k}(z),f^k(m))\notin D$. Therefore $(X,f)$ is eventually sensitive.

Now suppose that the system is eventually sensitive; let $D\in \mathscr{D}$ by an eventual-sensitivity entourage. Assume the system is equicontinuous. Since $X$ is compact the system is uniformly equicontinuous. Let $D_0$ be such that for any $x,y \in X$ if $(x,y) \in D_0$ then for any $n\in \mathbb{N}$, $(f^n(x),f^n(y)) \in D$. Let $x\in X$ be given. By eventual sensitivity there exists $n,k \in \mathbb{N}$ and $y \in D_0[f^n(x)]$ such that $(f^{n+k}(x),f^k(y)) \notin D$; this contradicts our assumption that the system is equicontinuous.
\end{proof}

\section{Proof of Theorem \ref{ThmExistsTransEvP}}\label{SectionProofOfExample}

Recursively define the finite words $C_n$ as follows. Let $C_0\coloneqq \texttt{10}$ and, for all $n\geq 1$, take \[C_n\coloneqq \texttt{1}^{8^n\abs{C_0 C_1 \ldots C_{n-1}}}\texttt{0}^{2^n\abs{C_0 {C_1} \ldots C_{n-1}}}.\]
For each $n\geq 1$ define 
\[Q_n\coloneqq \texttt{0}^{8^n\abs{C_0 C_1 \ldots C_{n-1}}}\texttt{0}^{2^n\abs{C_0 C_1 \ldots C_{n-1}}}.\]
Let $W_0 \coloneqq C_0Q_1$ and For each $n \geq 1$ let $W_n\coloneqq W_0W_1\ldots W_ {n-1}C_0C_1\ldots C_n Q_{n+1}$
(so $W_1=W_0C_0C_1Q_2$, $W_2= W_0W_1C_0C_1C_2Q_3$ and so on).

The first $8^n\abs{C_0 C_1 \ldots C_{n-1}}$ symbols of $C_n$ will be referred to as the $\texttt{1}$-part of $C_n$. Similarly, the last $2^n\abs{C_0 C_1 \ldots C_{n-1}}$ symbols of $C_n$ will be referred to as the $\texttt{0}$-part of $C_n$. We will refer to the word $C_0\ldots C_n Q_{n+1}$ as the \textit{closing segment} of $W_n$.

\begin{remark}
For any $n \in \mathbb{N}$, $\abs{C_n}=\abs{Q_n}$. We emphasise that $Q_n$ consists solely of $\texttt{0}$'s.
\end{remark}

To prove Theorem \ref{ThmExistsTransEvP} we will first need to prove the following lemma concerning the length of various words in our system. 

\begin{lemma}\label{lemmaClaim1}
For any $n \in {\mathbb{N}_0}$,
\begin{equation}\label{eqnClaim1}
6\left(8^{n+1}\abs{C_{n+1}}\right) \geq \abs{W_0W_1\ldots W_n C_0 \ldots C_{n+1}}+2\abs{W_n}.
\end{equation}
\end{lemma}

\begin{proof}
Let $P(n)$ be the statement
\[6\left(8^{n+1}\abs{C_{n+1}}\right) \geq \abs{W_0W_1\ldots W_n C_0 \ldots C_{n+1}}+2\abs{W_n}.\]

Case when $n=0$. Then $6(8^1\abs{C_1})=960$ whilst $\abs{W_0C_0C_1}+2\abs{W_0}=88$. Hence $P(0)$ holds.

Assume that $P(n)$ is true for all $n \leq k$ for some $k \in {\mathbb{N}_0}$. Will will prove $P(k+1)$ holds. For $P(k+1)$:


\begin{align*}
\text{RHS}&= \abs{W_0W_1\ldots W_k W_{k+1} C_0 \ldots C_{k+1} C_{k+2}} +2\abs{W_{k+1}} 
\\&= \abs{W_0W_1\ldots W_k C_0 \ldots C_{k+1}}+ 3\abs{W_{k+1}} +\abs{C_{k+2}}
\\&= 4\abs{W_0W_1\ldots W_k C_0 \ldots C_{k+1}} + 4\abs{Q_{k+2}} 
\tag*{as $\abs{Q_{k+2}}=\abs{C_{k+2}}$}
\\&= 4\abs{W_0W_1\ldots W_k C_0 \ldots C_{k+1}} + 4\left(8^{k+2}\abs{C_0\ldots C_{k+1}}\right) 
\\&\text{      }\hspace{0.3cm} + 4\left(2^{k+2}\abs{C_0\ldots C_{k+1}}\right) 
\\ &\leq 4\abs{W_0W_1\ldots W_k C_0 \ldots C_{k+1}} 
\\&\text{      } \hspace{0.3cm} + 4\left(8^{k+2}\abs{W_0 W_1 \ldots W_kC_0\ldots C_{k+1}}\right) 
\\&\text{      }\hspace{0.3cm}  + 4\left(2^{k+2}\abs{W_0W_1\ldots W_k C_0\ldots C_{k+1}}\right) 
\\ &\leq 6 \left(8^{k+2}\abs{W_0W_1\ldots W_k C_0\ldots C_{k+1}}\right) 
\tag*{as $2(8^{k+2}) \geq 4+4(2^{k+2})$}
\\ &\leq 6\left(8^{k+2}\big(6\big(8^{k+1}\abs{C_{k+1}}\big)\big)\right) 
\tag*{by the induction hypothesis}
\\ &\leq 6\left(8^{k+2}\big( 8^{k+2}\abs{C_{k+1}}\big)\right)
\\ &\leq 6\left(8^{k+2}\abs{C_{k+2}}\right) 
\tag*{by definition}
\\ &=\text{LHS.} \qedhere
\end{align*}
\end{proof}

\begin{remark}\label{RemarkExLengthOf1Part}
The length of the $\texttt{1}$-part of $C_{n+2}$ is $8^{n+2}\abs{C_0\ldots C_{n+1}}$. Notice that,
\begin{equation*}
8^{n+2}\abs{C_0\ldots C_{n+1}} \geq 6\left(8^{n+1}\abs{C_{n+1}}\right) + 2\left(8^{n+1}\abs{C_{n+1}}\right)
 > 6\left(8^{n+1}\abs{C_{n+1}}\right).
\end{equation*}
The final line is the LHS of Equation \ref{eqnClaim1}. By Lemma \ref{lemmaClaim1} this then means that the length of the $\texttt{1}$-part of $C_{n+1}$ is more than $2\left(8^{n+1}\abs{C_{n+1}}\right)$ greater than \abs{W_0W_1\ldots W_n C_0 \ldots C_{n+1}}+2\abs{W_n}. This observation will prove important later.
\end{remark}

\begin{corollary} For any $n,k \in {\mathbb{N}_0}$,
\begin{equation}\label{eqnCorollaryExampleEvP}
6\left(8^{n+1+k}\abs{C_{n+1+k}}\right) \geq \abs{W_0W_1\ldots W_{n+k} C_0\ldots C_{n+1+k}} +\sum_{i=0} ^{k-1} \abs{W_{n+1+i}}.
\end{equation}
\end{corollary}

\begin{proof}
Immediate from Lemma \ref{lemmaClaim1}.
\end{proof}


We now define a shift system $(X,\sigma)$ as follows. Let $x\coloneqq C_0C_1C_2C_3\ldots$ and $y \coloneqq W_0W_1W_2W_3\ldots$. Using the shift map $\sigma$ take 
\[X = \overline{\Orb_\sigma(x)\cup\Orb_\sigma(y) \cup\{\texttt{0}^nx, \texttt{0}^ny \mid n \in \mathbb{N}\}}.\]
Then $y$ is a transitive point in the system $(X,\sigma)$. Notice that $\texttt{0}^\infty,\texttt{10}^\infty \in X$ since they are in $\omega(x)$. We will show through a sequence of lemmas that the system $(X,f)$ satisfies the conditions in the theorem, in particular we will show $(x,\texttt{0}^\infty)$ is a non-trivial even continuity pair but $(x,\texttt{10}^\infty) \notin \EvP(X,f)$.

When working with dynamical systems, it can be helpful to visualise the forward orbit of a point as how it moves through time. In proving our claim we will use language like, `the first time $x$ visits $U \subseteq X$' or `when $x$ enters $U$ for the first time.' By such statements we mean, the least such $c \in {\mathbb{N}_0}$ such that $\sigma^c(x) \in U$. 
In similar fashion, we may speak of points travelling through words. For example, `When $x$ enters the $\texttt{0}$-part of $C_1$ for the first time, $y$ is travelling through $W_0$ for the first time; more specifically, $y$ is travelling through the $Q_1$-part of $Q_1 C_0$ for the first time.' This means that, if $t$ is such that $\sigma^t(x)$ is in the $\texttt{0}$-part of $C_1$ (i.e.\ $\left[\texttt{0}^4\right]$) for the first time, then there exists a unique $a \leq t$ such that $\sigma^a(y)\in W_0$ and $t-a < \abs{W_0}$. Similarly there exists a unique $b \leq t$ such that $\sigma^b(y)\in Q_1 C_0$ and $t-b <\abs{Q_1}$. In this particular example it can be seen that $t=18$, $a=0$ and $b=2$.

We introduce the following, \textit{first-hitting time}, notation. For $w \in X$ and $A \subseteq X$ such that $N(w,A) \neq \emptyset$,
\[\tau(w,A) \coloneqq \min N(w, A).\]
For example, $\tau(x,[C_2])=22$ whilst $\tau(y, [Q_1C_0])=2$. This allows us to translate long-winded sentences such as `$y$ enters $[Q_1C_0]$ for the first time before $x$ enters $[C_2]$ for the first time' into an equation, in this example:
\[\tau(y, [Q_1C_0])< \tau(x,[C_2]).\]


\begin{lemma}
$(x,\textup{\texttt{10}}^\infty) \notin \EvP(X,\sigma)$.
\end{lemma}

\begin{proof}
Let $O=[\texttt{10}]$; we claim this is an even-splitting neighbourhood of $\texttt{10}^\infty$ with regard to $x$. Let $U$ and $V$ be neighbourhoods of $x$ and $\texttt{10}^\infty$ respectively. Without loss of generality write $U=[C_0C_1C_2\ldots C_m]$ and $V=[\texttt{10}^l]$, where $m \geq l \geq 1$. There exists a point $p \in \Orb(y)$ such that $p \in [C_0C_1C_2\ldots C_mQ_{m+1}]$. Define $t\coloneqq \abs{C_0C_1 C_2 \ldots C_m}$ and note that $\sigma^t(x) \in [C_{m+1}]$, $\sigma^t(p) \in [Q_{m+1}]$. Let $k=8^{m+1}\abs{C_0C_1C_2\ldots C_m}$; this is the length of the $\texttt{1}$-part of $C_{m+1}$. It follows that 
\[\sigma^{t+k-1}(x) \in V,\]
and
\[\sigma^{t+k-1}(p) \in \left[\texttt{0}^{2^{m+1}}\right].\]
Hence $\sigma^{t+k-1}(p) \notin O$. Since $U$ and $V$ were picked arbitrarily this means $(x,\texttt{10}^\infty) \notin \EvP(X,\sigma)$. In particular $x \notin \Ev(X,\sigma)$.

\end{proof}

\begin{lemma}
There are no points of even continuity.
\end{lemma}
\begin{proof}
To see that $\Ev(X,\sigma)=\emptyset$, note that, since $X$ is compact, $\Ev(X,\sigma)=\Eq(X,\sigma)$ (see \cite[Theorem 7.23]{Kelley}). But since $(X,\sigma)$ is a shift space with no isolated points it is sensitive, hence $\Eq(X,\sigma)=\emptyset$.
\end{proof}

We will now set about showing that $(x,\texttt{0}^\infty)$ is a non-trivial even continuity pair. To do this we will need the following lemma. 

\begin{lemma}\label{lemma2ExampleEvP}
Let $n,a \in \mathbb{N}$ be such that $C_0\ldots C_n Q_{n+1}$ is an initial seqment of $z=\sigma^a(y)$. Then for any $k > n$, 
\[
\tau\left(x,[C_k]\right) \leq \tau \left(z, [Q_kC_0]\right) \leq 6\left(8^{k-1}\abs{C_{k-1}}\right)\]
\end{lemma}

In words, the first inequality means that $x$ enters $[C_k]$ for the first time no later than $z$ enters $[Q_kC_0]$ for the first time - which itself has happened by \textit{time} ``$6\left(8^{k-1}\abs{C_{k-1}}\right)$" by the second inequality. 

The final three inequalities emphasise that $z$ enters $[Q_kC_0]$ for the first time before $x$ enters the $\texttt{0}$-part of $C_k$ for the first time; in particular when $z$ enters $[Q_kC_0]$ for the first time $x$ still has to travel through at least $2\left(8^{k-1}\abs{C_{k-1}}\right)$ more $\texttt{1}$'s in the $\texttt{1}$-part of $C_k$ before it enters the $\texttt{0}$-part of $C_k$.

\begin{proof}

Let
\[n_0=\max \{c \in \mathbb{N} \mid \exists b < a : \sigma^b(y) \in [W_c]\}.\]

Note that $n_0$ is well defined and that $n_0 \geq n$. This means that $z$ is travelling through $W_{n_0}$ for the first time. 

Let $k> n$ be given. The first inequality follows immediately from the construction: The word $Q_kC_0$ appears in the sequence of $z$ for the first time only after the word $C_0\ldots C_{k-1}$. Similarly the word $C_k$ appears in the sequence of $x$ for the first time exactly after the word $C_0\ldots C_{k-1}$. Observing that $x=C_0C_1\ldots C_k C_{k+1}\ldots$ now gives the inequality, $\tau(x,[C_k]) \leq \tau(z, [Q_kC_0])$. 
It remains to show that the second inequality holds.

Let $z^\prime \in \Orb(y)$ be the point at which $y$ first enters $C_0\ldots C_n Q_{n+1}$; i.e. $z=\sigma^m(y)$ where $m= \tau \left(y, [C_0\ldots C_n Q_{n+1}] \right)$. Note that $z^\prime$ lies at the start of the closing segment of $W_n$. Indeed,
\[z^\prime = C_0 \ldots C_n Q_{n+1} W_{n+1} W_{n+2} W_{n+3} \ldots.\]
It is not difficult to see that $\tau \left(z, [Q_kC_0] \right) \leq \tau \left(z^\prime, [Q_kC_0] \right)$; it takes $z^\prime$ at least as long to enter $[Q_kC_0]$ for the first time as it does for $z$ to enter $[Q_kC_0]$ for the first time. (Observe that the letters (counting multiplicities) appearing in $z$ before the first appearance $Q_kC_0$ can be written as a list of words (including multiplicities) which also appear in $z^\prime$ (with multiplicities) before the first appearance of $Q_kC_0$ there. Hence the initial segment of $z^\prime$ up to the first appearance of $Q_kC_0$ is longer than that of the initial segment of $z$ up to the first appearance of $Q_kC_0$. 
We know $k > n$. First suppose that $k > n+1$. Then, by construction,
\begin{align*}
\tau \left(z^\prime, [Q_kC_0]\right) &= \abs{C_0\ldots C_n Q_{n+1}} + \left(\sum_{i=n+1} ^{k-2} \abs{W_i}\right) +\abs{W_0\ldots W_{k-2}C_0 \ldots C_{k-1}}
\\ &\leq \abs{W_0\ldots W_{k-2}C_0 \ldots C_{k-1}} + 2\abs{W_{k-2}}
\\ &\leq 6\left(8^{k-1}\abs{C_{k-1}}\right)\tag*{by Lemma \ref{lemmaClaim1}.}
\end{align*}
Since $\tau \left(z^\prime, [Q_kC_0] \right) \leq 6\left(8^{k-1}\abs{C_{k-1}}\right)$, and $\tau(z, [Q_kC_0]) \leq \tau(z^\prime, [Q_kC_0])$, we have that 
\[ \tau \left(z, [Q_kC_0] \right) \leq 6\left(8^{k-1}\abs{C_{k-1}}\right).\]

Now suppose that $k=n+1$. Then by Lemma \ref{lemmaClaim1}

\begin{align*}
\tau \left(z^\prime, [Q_kC_0] \right) &=\tau \left(x, [C_k]\right)
\\ &= \abs{C_0\ldots C_{k-1}}
\\ &\leq 6\left(8^{k-1}\abs{C_{k-1}}\right). 
\qedhere
\end{align*}

\end{proof}

\begin{corollary}\label{Corollary2ExampleEvP}
Let $n \in \mathbb{N}$ be such that $C_0\ldots C_n Q_{n+1}$ is an initial seqment of $z=\sigma^a(y)$ for some $a \in {\mathbb{N}_0}$. For any $k > n$, $x$ enters $[C_k]$ for the first time no later than $z$ enters $[Q_kC_0]$ for the first time. Additionally, $z$ enters $[Q_kC_0]$ for the first time before $x$ enters the $\textup{\texttt{0}}$-part of $C_k$ for the first time. In symbols:
\[\tau\left(x, [C_k]\right) \leq \tau\left(z, [Q_kC_0]\right) \leq \tau\left(x, \left[\textup{\texttt{0}}^{2^{k}\abs{C_0\ldots C_{k-1}}}\right]\right). \]
\end{corollary}

\begin{proof}
By Lemma \ref{lemma2ExampleEvP} it will suffice to show $6\left(8^{k-1}\abs{C_{k-1}}\right) \leq \tau\left(x, \left[\texttt{0}^{2^{k}\abs{C_0\ldots C_{k-1}}}\right]\right)$. Notice that $x$ has to travel through the $\texttt{1}$-part of $C_k$ before reaching $\left[\texttt{0}^{2^{k}\abs{C_0\ldots C_{k-1}}}\right]$. The length of the $\texttt{1}$-part of $C_k$ is $8^{k}\abs{C_0\ldots C_{k-1}}>6\left(8^{k-1}\abs{C_{k-1}}\right)$.
\end{proof}

\begin{corollary}
The ordered pair $(x, \textup{\texttt{0}}^\infty)$ is a non-trivial even continuity pair.
\end{corollary}

\begin{proof}
Since $\texttt{0}^\infty \in \omega(x)$, by definition $(x,\texttt{0}^\infty)$ is not a trivial even continuity pair. It thus suffices to show that $(x,\texttt{0}^\infty) \in \EvP(X,f)$, i.e.
\[\forall O\in \mathcal{N}_{\texttt{0}^\infty} \, \exists U \in \mathcal{N}_x \, \exists V \in \mathcal{N}_{\texttt{0}^\infty} : \forall n \in \mathbb{N}, \sigma^n(x) \in V \implies \sigma^n(U) \subseteq O.\]
Without loss of generality, let $O$ be the basic open neighbourhood $[\texttt{0}^n]$ of $\texttt{0}^\infty$. We claim $U=[C_0C_1\ldots C_n]$ and $V=[\texttt{0}^n]$ satisfy the even continuity condition. Since $\Orb(y)$ is dense and $O$, $U$ and $V$ are clopen, it suffices to consider only points in $U$ which are elements of the orbit of $y$. Let $z \in \Orb(y) \cap U$. Then $z \in [C_0 \ldots C_m Q_{m+1}]$ for some $m \geq n$. Suppose $l \in \mathbb{N}$ is such that $\sigma^l(x) \in V$.

\textbf{Case 1: $l \geq \tau\left(x, [C_{m+1}]\right)$.} Let $k \geq m+1$ be the greatest integer such that $l \geq \tau\left(x, [C_k]\right)$. It follows that at time $l$, $x$ is travelling through the $\texttt{0}$-part of $C_k$ for the first time, with at least $n$ $\texttt{0}$'s left to travel through. 
Furthermore, since $\tau\left(x, [C_k]\right) \leq \tau\left(z, [Q_kC_0]\right)$, and as $\abs{Q_k}=\abs{C_k}$, we have that $x$ finishes travelling though $C_k$ before $z$ finishes travelling through the $Q_k$-part of $Q_kC_0$. This means that at time $l$ there are at least as many $\texttt{0}$'s remaining in $Q_k$ (recall, $Q_k$ consists solely of $\texttt{0}$'s) for $z$ to travel through than there are $\texttt{0}$'s remaining in $C_k$ for $x$ to travel through. Since there are at least $n$ $\texttt{0}$'s left in $C_k$ for $x$ still to travel through (as $x \in V$), it follows that $z \in [\texttt{0}^n]=O$.

\textbf{Case 2: $l < \tau\left(x, [C_{m+1}]\right)$.} The initial segments of $x$ and $z$ are identical up to and including the first occurence of $C_{m+1}$. The word $C_{m+1}$ begins with a `$1$', therefore $l \leq \tau \left(x, [C_{m+1}]\right)-n$, because $\sigma^l(x) \in V=[\texttt{0}^n]$. In particular, it follows that $\sigma^l(z) \in [\texttt{0}^n]=O$.
\end{proof}

We will now set about showing that $\EqP(X,f) =\emptyset$. This will be completed in Lemma \ref{lemmaExEqPEmpty}. First we show that $y$ is not topologically equicontinuous with either one of the fixed points.

\begin{lemma}\label{lemmaEx01NotEqP}
Neither $(y,\textup{\texttt{0}}^\infty)$ nor $(y, \textup{\texttt{1}}^\infty)$ is an equicontinuity pair.
\end{lemma}
\begin{proof}
Recall that, to show that $(y,p) \notin \EqP(X,\sigma)$, where $p \in X$, we need to show that:
\[\exists O\in \mathcal{N}_{p} : \forall U \in \mathcal{N}_y \, \forall V \in \mathcal{N}_{p} \, \exists n \in \mathbb{N} : \sigma^n(U) \cap V \neq \emptyset \text{ and } \sigma^n(U) \not\subseteq O.\]

Let $O=[\texttt{0}]$. We claim $O$ is a splitting neighbourhood of $\texttt{0}^\infty$ with regard to $y$. Let $U \in \mathcal{N}_y$ and $V \in \mathcal{N}_{0^\infty}$ be given and let $[W_0\ldots W_n] \subseteq U$ and $[\texttt{0}^n]\subseteq V$. Let $m\in \mathbb{N}$ be such that $\texttt{0}^n$ appears as a subword of $C_m$; notice that it follows that $\texttt{0}^n$ is a subword of both $C_k$ and $W_k$ for all $k \geq m$. Let $l = \max \{n+2, m+2\}$. Notice that $2(8^{l-1}\abs{C_{l-1}}) >n+1$. 
Let $t=\tau\left(y,[W_l]\right)$ and write $z=\sigma^t(y)$. It follows that $z \in U$. It is worth comparing $z$ and $y$ side by side.
\[z=W_0W_1\ldots W_{l-1} C_0C_1 \ldots C_l Q_{l+1} W_{l+1}W_{l+2}\ldots,\]
and
\[y=W_0W_1\ldots  W_{l-1}W_0W_1\ldots  W_{l-1} C_0C_1 \ldots C_l Q_{l+1} W_{l+1}W_{l+2}\ldots.\]
Thus $z$ and $y$ share the same initial segment of $W_0\ldots W_{l-1}$. After this $z$ enters $[C_0C_1 \ldots C_l Q_{l+1}]$ for the first time whilst $y$ enters $[W_l]$ for the first time.
By Lemma \ref{lemmaClaim1},
\begin{equation}
6\left(8^{l-1}\abs{C_{l-1}}\right) \geq \abs{W_0W_1\ldots W_{l-2} C_0 \ldots C_{l-1}}.
\end{equation}
In particular the length of the $\texttt{1}$-part of $C_l$ is greater than $\abs{W_0W_1\ldots W_{l-2} C_0 \ldots C_{l-1}}+2\left(8^{l-1}\abs{C_{l-1}}\right)$. It follows that \[\tau\left(y, [Q_lC_0C_1\ldots C_l]\right) \leq \tau\left(z, \left[\texttt{0}^{2^{l}\abs{C_0\ldots C_{l-1}}}Q_{l+1}\right]\right) - 2\left(8^{l-1}\abs{C_{l-1}}\right).\]
That is, $y$ enters $[Q_lC_0C_1\ldots C_l]$ for the first time before $z$ enters the $\texttt{0}$-part of $[C_l]$ for the first time; in particular when $y$ enters $[Q_lC_0C_1\ldots C_l]$ for the first time $z$ still has to travel through at least $2\left(8^{l-1}\abs{C_{l-1}}\right)$ more $\texttt{1}$'s in the $\texttt{1}$-part of $C_l$ before it enters the $\texttt{0}$-part of $C_l$. Since $\tau\left(z, [C_l]\right) \leq \tau\left(y, [Q_lC_0C_1\ldots C_l]\right)$ we get that 
\[\sigma^{\tau\left(y, [Q_lC_0C_1\ldots C_l]\right)}(y) \in V\]
but
\[\sigma^{\tau\left(y, [Q_lC_0C_1\ldots C_l]\right)}(z) \in \left[\texttt{1}^{2\left(8^{l-1}\abs{C_{l-1}}\right)}\right] \implies \sigma^{\tau\left(y, [Q_lC_0C_1\ldots C_l]\right)}(z) \notin O.\]
Hence $(y,0^\infty) \notin \EqP(X,\sigma)$. (Indeed, we have actually shown the stronger claim that $(y,\texttt{0}^\infty) \notin \EvP(X,\sigma)$.)

Now let $O=[\texttt{1}]$. We claim $O$ is a splitting neighbourhood of $\texttt{1}^\infty$ with regard to $y$. Let $U \in \mathcal{N}_y$ and $V \in \mathcal{N}_{\texttt{1}^\infty}$. Let $[W_0\ldots W_n] \subseteq U$ and $[\texttt{1}^n]\subseteq V$. Let $m\in \mathbb{N}$ be such that $\texttt{1}^n$ appears as a subword of $C_m$; notice that it follows that $\texttt{1}^n$ is a subword of both $C_k$ and $W_k$ for all $k \geq m$. Let $l = \max \{n+2, m+2\}$. Notice that $2(8^{l-1}\abs{C_{l-1}}) >n+1$. 
Let $t=\tau\left(y,[W_l]\right)$ and write $z=\sigma^t(y)$. It follows that $z \in U$. As before, $z$ and $y$ share the same initial segment of $W_0\ldots W_{l-1}$. After this $z$ enters $[C_0C_1 \ldots C_l Q_{l+1}]$ for the first time whilst $y$ enters $[W_l]$ for the first time.
By an almost identical argument to the one we used in the previous paragraph (whilst showing that $p\neq \texttt{0}^\infty$), we know that
\[\sigma^{\tau\left(y, [Q_lC_0C_1\ldots C_l]\right)}(z) \in \left[\texttt{1}^{2\left(8^{l-1}\abs{C_{l-1}}\right)}\right] \subseteq V.\]
However
\[\sigma^{\tau\left(y, [Q_lC_0C_1\ldots C_l]\right)}(y) \in [\texttt{0}] \subseteq X\setminus O.\]
Hence $(y,\texttt{1}^\infty) \notin \EqP(X,\sigma)$.
\end{proof}

\begin{lemma}\label{lemmaExEqPEmpty}
The system $(X,\sigma)$ has no equicontinuity pairs.
\end{lemma}

\begin{proof}

By Remark \ref{RemarkEqPairImpliesTransPtHasEqPartner} it will suffice to show that $(y,p) \notin \EqP(X,\sigma)$ for any $p \in X$. We need to show that, for any $p \in X$,
\[\exists O\in \mathcal{N}_{p} : \forall U \in \mathcal{N}_y \, \forall V \in \mathcal{N}_{p} \, \exists n \in \mathbb{N} : \sigma^n(U) \cap V \neq \emptyset \text{ and } \sigma^n(U) \not\subseteq O.\]

Suppose that $(y,p) \in \EqP(X,\sigma)$; write $p=p_0p_1p_2\ldots$. 
By Lemma \ref{lemmaEx01NotEqP} we have that $p \notin \{\texttt{0}^\infty, \texttt{1}^\infty\}$. This means that there exist $i,j \in {\mathbb{N}_0}$ such that $p_i=\texttt{0}$ and $p_j=\texttt{1}$. Fix such an $i$ and a $j$ and take $k \geq \max \{i,j\}$. Let $O=[p_0p_1\ldots p_k]$. We claim $O$ is a splitting neighbourhood of $p$ with regard to $y$. Let $U \in \mathcal{N}_y$ and $V \in \mathcal{N}_p$ be given and let $[W_0\ldots W_n] \subseteq U$ and $[p_0p_1\ldots p_n]\subseteq V$; without loss of generality $n\geq k$. Let $m\in \mathbb{N}$ be such that $p_0p_1\ldots p_n$ appears as a subword of $W_m$; notice that it follows that $p_0p_1\ldots p_n$ is a subword of $W_a$ for all $a \geq m$. Let $l \geq \max \{n+2, m+2\}$ be such that $2(8^{l-1}\abs{C_{l-1}}) >n+1$. 
Let $t=\tau\left(y,[W_l]\right)$ and write $z=\sigma^t(y)$. It follows that $z \in U$. It is worth comparing $z$ and $y$ side by side.
\[z=W_0W_1\ldots W_{l-1} C_0C_1 \ldots C_l Q_{l+1} W_{l+1}W_{l+2}\ldots,\]
and
\[y=W_0W_1\ldots  W_{l-1}W_0W_1\ldots  W_{l-1} C_0C_1 \ldots C_l Q_{l+1} W_{l+1}W_{l+2}\ldots.\]
Notice $z$ and $y$ share the same initial segment given by $W_0\ldots W_{l-1}$. After this $z$ enters $[C_0C_1 \ldots C_l Q_{l+1}]$ for the first time whilst $y$ enters $[W_l]$ for the first time.
Notice that, for all $i \in {\mathbb{N}_0}$, $\abs{W_i} \geq \abs{Q_{i+1}}=\abs{C_{i+1}}$. In addition $\abs{W_0} \geq \abs{C_0C_1}$. It follows that 
\begin{equation}\label{EqnExyz1}
\tau\left(z,[C_l]\right) \leq \tau\left(y,[W_{l-1}C_0\ldots C_lQ_{l+1}]\right).
\end{equation}
Observe,
\[\tau\left(y, [Q_l C_0\ldots C_lQ_{l+1}]\right)=\abs{W_0\ldots W_{l-1}}+\abs{W_0\ldots W_{l-2}} +\abs{W_0\ldots W_{l-2}C_0\ldots C_{l-1}}.\]
Similarly observe
\[\tau\left(z, [C_l]\right)=\abs{W_0\ldots W_{l-1}} +\abs{C_0\ldots C_{l-1}}. \]
Therefore,
\begin{align*}
\tau\left(y, [Q_l C_0\ldots C_lQ_{l+1}]\right)-\tau\left(z, [C_l]\right) &= 2\abs{W_0\ldots W_{l-2}},
\\ &\leq \abs{W_0W_1\ldots W_{l-2}} + 2\abs{W_{l-2}},
\\ &\leq 6\left(8^{l-1}\abs{C_{l-1}}\right)\tag*{by Lemma \ref{lemmaClaim1}.}
\end{align*}
Thus
\begin{equation}\label{EqnExzy2}
\tau\left(z,[C_l]\right) +6\left(8^{l-1}\abs{C_{l-1}}\right) \geq \tau\left(y, [Q_l C_0\ldots C_lQ_{l+1}]\right).
\end{equation}
Putting inequalities (\ref{EqnExyz1}) and (\ref{EqnExzy2}) together we obtain:
\begin{align*}
\tau\left(z,[C_l]\right) &\leq \tau\left(y,[W_{l-1}C_0\ldots C_lQ_{l+1}]\right)
\\ &\leq  \tau\left(y, [Q_l C_0\ldots C_lQ_{l+1}]\right)
\\ &\leq \tau\left(z,[C_l]\right) +6\left(8^{l-1}\abs{C_{l-1}}\right)
\\ &\leq  \tau\left(z, \left[\texttt{0}^{2^{l}\abs{C_0\ldots C_{l-1}}}Q_{l+1}\right]\right) - 2\left(8^{l-1}\abs{C_{l-1}}\right).
\end{align*}
The final inequality follows because, by definition, the length of the $\texttt{1}$-part of $C_l$ is more than $8^{l}\abs{C_{l-1}}$.
It follows that, whilst $y$ enters $W_0W_1\ldots W_{l-2} C_0 \ldots C_{l-1}$ for the second time, $z$ is travelling through the $\texttt{1}$-part of $C_l$. When $y$ finishes travelling though $W_0W_1\ldots W_{l-2} C_0 \ldots C_{l-1}$ for the second time (and enters $[Q_l C_0\ldots C_lQ_{l+1}]$ for the first time), $z$ still has to travel through at least $2\left(8^{l-1}\abs{C_{l-1}}\right)$ more $\texttt{1}$'s in the $\texttt{1}$-part of $C_l$ before it enters the $\texttt{0}$-part of $C_l$. Because $p_0\ldots p_n$ is a subword of $W_{l-2}$, which is a subword of $W_0W_1\ldots W_{l-2} C_0 \ldots C_{l-1}$, and since $[p_0p_1\ldots p_n]\subseteq V$ it follows that $y$ enters $V$ whilst travelling through $W_0W_1\ldots W_{l-2} C_0 \ldots C_{l-1}$ for the second time. Take $c \in {\mathbb{N}_0}$ such that $\sigma^c(y) \in V$ where $c>\tau\left(y, [W_l]\right)$ and $c < \tau\left(y, [Q_l C_0\ldots C_lQ_{l+1}]\right)$. Since $2\left(8^{l-1}\abs{C_{l-1}}\right)>n+1$ it follows that $\sigma^c(z) \in [1^{n+1}]$. But the word inducing $O$ (i.e. $p_0\ldots p_k$) contains at least one $\texttt{0}$ and $n+1\geq k+1$. Hence $\sigma^c(z) \notin O$; in particular $\sigma^c(U) \cap V \neq \emptyset$ and $\sigma^c(U) \not\subseteq O$.
\end{proof}

\begin{lemma}
The system $(X,\sigma)$ is Auslander-Yorke chaotic but not Devaney chaotic.
\end{lemma}
\begin{proof}
The system is both transitive and sensitive, this means it is Auslander-Yorke chaotic. It may be verified that the only periodic points are $\texttt{0}^\infty$ and $\texttt{1}^\infty$, hence the system is not Devaney chaotic.
\end{proof}

\begin{lemma}
The system $(X,\sigma)$ is both strongly and densely Li-Yorke chaotic.
\end{lemma}

\begin{proof}
By Corollary 7.3.7 in \cite{deVries}, a compact metric system without isolated points is both strongly and densely Li-Yorke chaotic if the system is transitive and there is a fixed point. Since our system satisfies these conditions the result follows.
\end{proof}

\bibliographystyle{plain} 
\bibliography{bib}

\end{document}